\documentclass{aims}
\usepackage{amsmath}
  \usepackage{paralist}
  \usepackage{graphics} 
  \usepackage{epsfig} 
\usepackage{graphicx}  \usepackage{epstopdf}
 \usepackage[colorlinks=true]{hyperref}
\hypersetup{urlcolor=blue, citecolor=red}

\renewcommand{\appendixname}{Appendix}
\newtheorem{theorem}{Theorem}[section]

\theoremstyle{definition}
\newtheorem{definition}[theorem]{Definition}
\newtheorem{lemma}[theorem]{Lemma}
\newtheorem{proposition}{Proposition}

\usepackage{enumerate}
\usepackage{ amssymb }
\usepackage{cancel}

\let\e=\varepsilon

\let\p=\partial

\let\O=\Omega

\let\o=\omega

\let\L=\Lambda

\let\hide\iffalse
\let\unhide\fi

\newcommand{\R}{\mathbb{R}}

\newcommand{\be}{\begin{equation}}
\newcommand{\bm}{\begin{multline}}
\newcommand{\ee}{\end{equation}}
\newcommand{\dd}{\mathrm{d}}

\newcommand{\Bes}{\begin{eqnarray*}}
\newcommand{\Ees}{\end{eqnarray*}}
\newcommand{\Be}{\begin{equation} }
\newcommand{\Ee}{\end{equation}}
\newcommand{\Bs}{\begin{split}}
\newcommand{\Es}{\end{split}}

\def\p{\partial}

\def\O{\Omega}
\def\R{\mathbb{R}}

\def\B{\begin{equation}}
\def\E{\end{equation}}
\def\BN{\begin{eqnarray*}}
\def\EN{\end{eqnarray*}}

\title[Glassey-Strauss representation in a Half Space] 
{Glassey-Strauss representation of Vlasov-Maxwell systems in a Half Space} 
 
\author{Yunbai Cao and Chanwoo Kim}


\email{yc1157@math.rutgers.edu}
\email{ckim.pde@gmail.com; chanwoo.kim@wisc.edu}

\thanks{CK is supported in part by National Science Foundation under Grant No. $1900923$ and the Brain Pool program (NRF-2021H1D3A2A01039047) of the Ministry of Science and ICT in Korea.}

\begin{document}

\maketitle

\centerline{\scshape Yunbai Cao}
\medskip
{\footnotesize
 \centerline{Department of Mathematics}
   \centerline{Rutgers University}
   \centerline{Piscataway, NJ 08854, USA}
} 

\medskip

\centerline{\scshape Chanwoo Kim}
\medskip
{\footnotesize
 \centerline{Department of Mathematics}  
 \centerline{University of Wisconsin-Madison}
    \centerline{Madison, WI 53717, USA}
    
     \centerline{$\&$} 
 \centerline{Department of Mathematical Sciences}  
 \centerline{Seoul National University}
    \centerline{Seoul 08826, Korea}

}
\bigskip

 \begin{center}
\textit{This paper is dedicated to the memory of the late Bob Glassey.} 
\end{center}

\begin{abstract}
Following closely the classical works \cite{Glassey}-\cite{GSc2} by Glassey, Strauss, and Schaeffer, we present a version of the Glassey-Strauss representation for the Vlasov-Maxwell systems in a 3D half space when the boundary is the perfect conductor.
\end{abstract}

 \section{Vlasov-Maxwell systems}
Consider the plasma particles of several species with masses $m_\beta$ and charges $e_\beta$ for $\beta=1,2,\cdots ,N,$ which occupy the half space 
 \Be\label{defO}
 \O = \mathbb R^3_+ = \{(x_1,x_2,x_3) \in \R^3: x_3>0\} \ni x. 
 \Ee
  The relativistic velocity for each particle is, for the speed of light $c$,
\Be\label{rel_v}
\hat v_\beta = \frac{v}{\sqrt{m_\beta^2+ |v|^2/c^2}} \ \ \ \text{for}   \ \  v \in \R^3.
\Ee 
Denote by $f_\beta(t,x,v)$ the particle densities of the species. The total electric charge density (total charge per unit volume) $\rho$ and 
the total electric current density (total current per unit area) $J$ are given by 
\begin{align}
\rho(t,x) &=   \int_{\mathbb R^3} \sum_{\beta}e_\beta f_\beta (t,x,v) dv,\label{charge} \\
J(t,x)&=  \int_{\mathbb R^3} \sum_{\beta} \hat v_\beta e_\beta f_\beta(t,x,v) dv.\label{current}
\end{align}

The relativistic Vlasov-Maxwell system governs the evolution of $f_\beta(t,x,v)$ (see page 140 of the Glassey's book \cite{Glassey}): for $(t,x,v) \in [0,T] \times \O \times \R^3$,
 \Be \label{VMfrakF}
\p_t f_\beta + \hat v_\beta \cdot \nabla_x f_\beta+  (e_\beta E + e_\beta\frac{\hat v_\beta}{c} \times B    - m_\beta g \mathbf e_3  ) \cdot \nabla_v f_\beta =0, 
\Ee
where $g$ is the gravitational constant (we can easily treat a general given external field). The electromagnetic fields $(E,B)$ is determined by the Maxwell's equations in a vacuum (in Gaussian units)
 \begin{align}
  \nabla_x \cdot E &=  4\pi  \rho,\label{divEfield}\\
  \nabla_x \times E&= -\frac{1}{c}  \p_t B  , \label{curlfieldE}\\
     \nabla_x \cdot B &= 0, \label{divBfield}\\
  \nabla_x \times B &= \frac{4 \pi}{c} J +  \frac{1}{c} \p_t E   .\label{curlfieldB} \end{align}

\section{Boundaries} Plasma particles can face various forms of boundaries in different scales from the astronomic one to the laboratory (\cite{CKL,CK}). In particular, we are interested in the plasma inside the fusion reactors in this paper. So-called plasma-facing materials, the materials that line the vacuum vessel of the fusion reactors, experience violent conditions as they are subjected to high-speed particle and neutron flux and high heat loads. These require several challenging conditions for the boundary materials, namely high thermal conductivity for efficient heat transport, high cohesive energy for low erosion by particle bombardment, and low atomic number to minimize plasma cooling. Traditionally sturdy metals and alloys such as stainless steel, tungsten, titanium, beryllium, and molybdenum have been used for the boundary material \cite{Federici}. As these metals have very high electric conductivity, we can regard them as \textit{the perfect conductor}. This boundary condition is the major interest of the paper (see Section \ref{PCB}). 

On the other hand, carbon/carbon composites such as refined graphite have excellent thermal and mechanical properties: eroded
carbon atoms are fully stripped in the plasma core, thus reducing core radiation. In addition,
the surface does not melt but simply sublimes if overheated. For this reason, the majority of the latest machines have expanded graphite coverage tile to include all of the vacuum vessel walls \cite{Federici}. Graphite is an allotrope of carbon, existing as the collection of thin layers of a giant carbon atoms' covalent lattice. As there is one delocalized electron per carbon atom, graphite does conduct electricity throughout each layer of the graphite lattice but poorly across different layers. Due to the anisotropic electric conductivity of graphite, one has to employ different boundary conditions from one for the metal boundary.

\subsection{Perfect conductor boundary}\label{PCB}
In this section, we consider the boundary conditions of $(E,B)$ at the boundary $\p\O:=\{(x_1,x_2,x_3) \in \R^3: x_3=0\}$. For that, actually we consider more general situation: two different media occupy $\mathbb R^3_+$ and $\mathbb R^3_-:=   \{(x_1,x_2,x_3) \in \R^3: x_3<0\}$ separately. In case that media are subject to electric and
magnetic polarization, it is much more convenient to write the Maxwell's equations only for the free charges and free currents in terms of SI units (see Chapter 7 in \cite{Griffiths}): 
 \begin{align}
  \nabla_x \cdot D &=   \rho_{\text{free}},\label{divEfield-}\\
  \nabla_x \times E&= -  \p_t B  , \label{curlfieldE-}\\
     \nabla_x \cdot B &= 0, \label{divBfield-}\\
  \nabla_x \times H &=    J_{\text{free}} +    \p_t D   ,\label{curlfieldB-} \end{align} 
where $\epsilon_0$ is the permittivity of free space and $\mu_0$ is the permeability of free space (note that the speed of light $c= \frac{1}{\sqrt{\epsilon_0 \mu_0}}$). Here, $D= \epsilon_0 E + P$ and $H= \frac{1}{\mu_0}B-M$, while an electric polarization $P$ and a magnetic polarization $M$ are determined by appropriate constitutive
relations in terms of $E$ and $B$. For example, a linear medium has 
\Be\label{linear_M}
P= \e_0 \chi_e E,   \ \ \ M=  \chi_m H.
\Ee
Here,  $\chi_e$ and $\chi_m$ are called the electric susceptibility and magnetic susceptibility, respectively. In a vacuum, as $\chi_e=0= \chi_m$ and $\rho=\rho_{\text{free}}$ and $J=J_{\text{free}}$ (all plasma particles/charges are free to move), we recover \ref{divEfield}-\ref{curlfieldB}.

Denote by $n$ the outward unit normal of $\O$ (which is $n=-\mathbf e_3$ for our case); $[V]$ the jump of $V$ across $\p\O$: $[V](x_1,x_2)=  \lim_{x_3 \downarrow 0}V(x_1,x_2,x_3)- \lim_{x_3 \uparrow 0}V(x_1,x_2,x_3)$. Then from \ref{curlfieldE-} and \ref{divBfield-} we derive the jump conditions
\Be\label{jump_EB}
 n \times  [E]  = 0 , \ \ \ 
n \cdot [B]    =0.
\Ee
In other words, the tangential electric fields $E_1$, $E_2$, and the normal magnetic field $B_3$ are continuous across the interface $\p\O$. We note that \ref{jump_EB} hold in general, no matter what constitutive relations hold (\cite{CS,GuoVM}). (In special circumstances (e.g. electromagnetic band gap structures), one has to consider a non-zero surface magnetic charge and current, in which \ref{jump_EB} should be replaced by discontinuous jump conditions \cite{Surface_EM}. Such cases are out of our interest in the paper.)

Now we come back to the original situation that the plasma particles stay in a vacuum of the upper half space $\O=\R^3_+$, while some matter fills the lower half space $\R^3_-$. We assume that the current follows the Ohm's law in the matter:
\Be\label{Ohm}
J_{\text{free}}= \sigma  \{\text{Lorentz force}\}    , \Ee 
where Lorentz force equals $e_\beta E + e_\beta \frac{\hat{v}_\beta}{c} \times B$ as the gravitation effect is negligible inside the matter. Here, $\sigma$ is the conductivity of the matter, which equals the reciprocal of the resistivity. The perfect conductors have $\sigma=\infty$ and the dielectrics get $\sigma=0$, while most of real matter is between them. As the drift speed of electrons/ions in the matter is slow (typical drift speed of electrons is few millimeters per second), we ignore the magnetic effect in the Lorentz force to derive that $\nabla_x \cdot J_{\text{free}}= \sigma  \nabla_x \cdot E$. We assume that the matter is the linear medium \ref{linear_M} and hence $D= \epsilon_0 (1+ \chi_e) E $. We derive that, from the continuity equation and \ref{divEfield-},
\Be\notag
\p_t \rho_{\text{free}} = -\nabla_x \cdot J_{\text{free}}= - \sigma \nabla_x \cdot E = - \frac{\sigma}{ \epsilon_0 (1+ \chi_e) } \nabla_x \cdot D =  - \frac{\sigma}{ \epsilon_0 (1+ \chi_e) } \rho_{\text{free}}. 
\Ee
Hence the charge density $\rho_{\text{free}}$ vanishes in the time scale of $1/\sigma$, which implies $\rho_{\text{free}}\equiv 0$ inside the perfect conductor $(\sigma=\infty)$. As a consequence, the charge density and current density accumulate only on the surface/boundary/interface (``Skin effect" \cite{Jackson}). Moreover, \ref{curlfieldB-} and \ref{Ohm} formally imply that $E\equiv 0$, and then \ref{curlfieldE-} forces $\p_tB=0$ inside the perfect conductor. Therefore by assuming the initial datum of $B$ vanishes in $\R^3_-$, we have $B\equiv 0$ in $\R^3_-$. On the other hand, the superconductor has $B\equiv0$ no matter what initial datum of $B$ is (the Meissner effect). 

We summarize the above discussion about $E_1,E_2, B_3$ in \ref{PCBC} and will derive the boundary conditions for $E_3$ and $B_1,B_2$ using the equations:
\begin{definition}[Perfect conductor (or superconductor) boundary condition]\label{def:PCBC}
Assume the lower half space $\R^3_-$ consists of a linear medium \ref{linear_M} of the perfect conductor $\sigma=\infty$ satisfying the Ohm's law \ref{Ohm}. We further assume either the initial magnetic field $B$ totally vanishes or the matter of $\R^3_-$ is the superconductor. Then $E\equiv 0 \equiv B$ in $\R^3_-$. Therefore, from \ref{jump_EB} we derive boundary conditions of the solutions $(E,B)$ to \ref{divEfield}-\ref{curlfieldB}:
\Be
E_1=0=E_2, \ \ \ B_3=0
\ \ \  \text{on} \ \ \ \p\O.\label{PCBC} 
\Ee
Moreover,
\begin{align}
\p_3 E_3 =  4\pi \rho \ \ \ &\text{on} \ \ \ \p\O,\label{BC:E_3} \\
\p_3 B_1 = 4\pi J_2, \ \   \p_3 B_2 = - 4\pi J_1 \ \ \ &\text{on} \ \ \ \p\O.\label{BC:B||}
\end{align}
\end{definition}

We only need to derive the boundary conditions for $E_3$, $B_1$, and $B_2$. We achieve them by using the equations \ref{divEfield} and \ref{curlfieldB}. From \ref{divEfield}, we have $
\p_3 E_3 = - \p_1 E_1 - \p_2 E_2 + 4\pi \rho.$ Then \ref{PCBC} formally implies \ref{BC:E_3}. 

Now from \ref{curlfieldB}, we have, for $n= -\mathbf e_3$, 
\Be \notag
\frac{1}{c} n \times   \p_t E   + \frac{ 4\pi}{c}    n \times J    - n \times ( \nabla_x \times B) = 0.
\Ee
From \ref{PCBC}, on $\p\O$ we deduce that $n \times   \p_t E =0$ and hence \Be
\begin{split}\notag
0 =    4 \pi \begin{bmatrix} J_2 \\ -J_1 \\ 0 \end{bmatrix}  -  \begin{bmatrix} 0 \\ 0 \\ - 1 \end{bmatrix} \times  \begin{bmatrix} \p_2 B_3 - \p_3 B_2 \\ - ( \p_1 B_3 - \p_3 B_1 ) \\ \p_1 B_2 - \p_2 B_1 \end{bmatrix}  
  =    4 \pi \begin{bmatrix} J_2 \\ -J_1 \\ 0 \end{bmatrix} - \begin{bmatrix} \p_3 B_1 \\  \p_3 B_2 \\ 0 \end{bmatrix} .
\end{split}
\Ee
Therefore, we conclude \ref{BC:B||}.

\subsection{Surface charge and surface current}\label{Surface_EM}To consider general jump conditions across the interface \ref{jump_EB2}, we need to count a surface charge with density $\sigma_{\text{free}}$ and a surface current with density $K_{\text{free}}$ which are ``concentrated" on the interface $\p\O$ (see \cite{Jackson, Griffiths}). 
 Physically, a non-zero surface charge and current exist on the surface of the perfect conductor as the interior electric field is zero (see a survey on the concept of the ``perfect'' conductor and the surface charge and current in history \cite{McDonald}). Then from \ref{curlfieldB} and \ref{divEfield}, we formally get
\Be\label{jump_EB2}
n \times [H]  =n  \times K_{\text{free}}  ,\ \ \    
n \cdot  [D] = \sigma_{\text{free}}. 
\Ee
For example, if both media are linear \ref{linear_M} then \ref{jump_EB2} implies that 
\Be\label{jump_EB3}
n \times (\frac{1}{\mu_+} B_+
-  \frac{1}{\mu_-} B_-)=n  \times K_{\text{free}}  ,\ \ \ 
n \cdot   (\epsilon_+ E_+ - \epsilon_- E_-) = \sigma_{\text{free}}, 
\Ee
where $\epsilon_\pm= \epsilon_0 (1+ \chi_{e,\pm})$ and $\mu_\pm= \mu_0 (1+ \chi_{m,\pm})$ are the permittivity and permeability for the upper and lower media. In the case of Definition \ref{def:PCBC}, the upper half is the vacuum and the lower half is a perfect conductor with $B\equiv 0$, then \ref{jump_EB2} implies that 
\Be\label{Ksigma_pO}
K_{\text{free}\parallel} = 
\frac{1}{\mu_0 }B_\parallel   |_{\p\O} , \ \ \  \ 
 \sigma_{\text{free}}= \e_0 E_3|_{\p\O} .
\Ee
We note that \ref{Ksigma_pO} is not the boundary condition, but one can measure the surface charge and surface current on the surface of the perfect conductor by evaluating $E$ and $B$. 

On the other hand, unless the dielectric media can be polarized on the interface, both surface charge and current vanish on the interface. This results in the dielectric boundary condition, which is \ref{jump_EB2} with $K_{\text{free}}=0=\sigma_{\text{free}}$ (\cite{CS}). When the media have anisotropic conductivities as graphite, the surface charge and current would not be prescribed simply but determined by PDEs.

\hide

and let $\p D_+$ be the upper face of the box $D$ and $\p D_-$ be the lower face. Then integrate \ref{divEfield} over $D$ and apply the divergence theorem, we get
\[
\int_{ \p D_+ } E_3 - \int_{\p D_- } E_3 = \int_{ \{ x \in \p \O: |x_1 | < l , |x_2 | < l  \}}  \sigma + O(\e),
\]
where we also put the surface integration of $E_3$ on the $4$ side faces of $D$ into $O(\e)$.

 Letting $\e \ll 1 $ and ignoring the term we get
\Be
E_3^+ - E_3^- = \sigma,
\Ee
where $E_3^+ - E_3^-$ is the jump of $E_3$ across the surface of the conductor $\p \O$. And by the same reason we can integrate \ref{divBfield} over $D$ to get
\Be
B_3^+ - B_3^- = 0.
\Ee
Next we consider a small loop straddling the surface
\[
\mathcal P = \{ x: x_1 = 0, |x_2 | < l,  |x_3 | < \e \},
\]
and let $\mathcal S$ be the surface bounded by the closed loop. Let $\ell_+ $ be the upper side line of the loop and $\ell_-$ be the lower side line. Again, since the currents are ``concentrated" on the surface, we have 
\[
\int_{\mathcal P} 4 \pi j  `` = "  \int_{ \{ x \in \p \O: x_1 =0, |x_2| < l \} } (\mathbf K \times n )_2 +  O(\e)
\]

Then integrate \ref{curlfieldB} over $\mathcal S$ and apply the Stokes' theorem we get
\[
\int_{\ell_+} B_2 - \int_{\ell_-}  B_2 =  \int_{ \{ x \in \p \O: x_1 =0, |x_2| < l \} } (\mathbf K \times n )_2 +  O(\e),
\]
so $B_2^+ - B_2^- = \mathbf (K \times n)_2$. And in the same way $B_1^+ - B_1^- = (\mathbf K \times n)_1$. Therefore we have
\Be
B_\parallel^+ - B_\parallel^- = \mathbf K \times n.
\Ee

Similarly, integrate \ref{curlfieldE} over $\mathcal S$ gives
\[
\int_{\ell_+} E_2 - \int_{\ell_-}  E_2 = O(\e) 
\]
Letting $\e \ll 1 $ we have $E_2^+ - E_2^- = 0$. And the same happens for $E_1$. Therefore 
\Be
E_\parallel^+ - E_\parallel^- = 0.
\Ee

In summary, we have on the boundary $\p \O$,
\Be
\begin{split}
E_3^+ - E_3^- = & \sigma,  \, E_\parallel^+ - E_\parallel^- =  0.
\\  B_3^+ - B_3^- = & 0,  \,  B_\parallel^+ - B_\parallel^- = \mathbf K \times n,
\end{split}
\Ee
where $\rho$ is the surface charge density and $\mathbf K$ is the surface current density. And if the material in the lower half space is perfect conductor, i.e. the conductivity of the material is $\sim \infty$, then $E^- = 0$. And thus $\p_t B^- = - \nabla \times E^- = 0$. Thus if $B^-(0) = 0$, $B^-(t) = 0$ for all $t$. Or if we assume the lower half space is superconductor then the electromagnetic field vanishes inside it, i.e. $E^- = 0$, and $B^- = 0 $. Then we have on the boundary $\p \O$,
\Be
\begin{split}
E_3^+  = & \sigma,  \, E_\parallel^+  =  0.
\\  B_3^+  = & 0,  \,  B_\parallel^+  = \mathbf K \times n. 
\end{split}
\Ee


\unhide

\hide
 the Maxwell's equations in a vacuum (in SI units):  
 \begin{align}
  \nabla_x \cdot E &=  \frac{1}{\epsilon_0} \rho,\label{divEfield}\\
  \nabla_x \times E&= -  \p_t B  , \label{curlfieldE}\\
     \nabla_x \cdot B &= 0, \label{divBfield}\\
  \nabla_x \times B &= \mu_0 J + \mu_0 \epsilon_0 \p_t E   .\label{curlfieldB} \end{align} 

\unhide

\section{The Glassey-Strauss representation in $\R^3$ (\cite{GS})}\label{section_GS}
In the whole space, $E$ and $B$ solve
 \begin{align}   
\p_t^2 E - \Delta_x E  &= - 4 \pi \nabla_x \rho -   4 \pi\p_t J,   \label{wave_eq_E}
\\
\p_t^2 B - \Delta_x B  &=  4 \pi  \nabla_x \times J , \label{wave_eq_B}
 \end{align} 
 with the initial data
 \Be\label{initialC} 
 \begin{split}
 E|_{t=0} = E_0,\ \ \p_t E|_{t=0} = \p_t E_0:= \nabla_x \times B_0 - 4\pi J|_{t=0}
  ,
  \\
 B|_{t=0} = B_0,\ \ \p_t B|_{t=0} = \p_t B_0:=  - \nabla_x \times E_0 
  .
 \end{split}
 \Ee 
 Obviously the wave equations suffer from the ``loss of derivatives'' of $(E,B)$ with respect to the regularity of the source terms $\rho$ and $J$. As Glassey mentions in his book \cite{Glassey}, the key idea of the Glassey-Strauss representation is replacing the derivatives $\p_t, \nabla_x$ by a geometric operator $T$ in \ref{def:T} and a kinetic transport operator $S$ in \ref{def:S}:
  \Be \label{pxptST}
\begin{split} 
 \p_t  = \frac{S- \hat{v} \cdot T}{1+ \hat v \cdot \o}, \ \ 
\p_i  = \frac{\o_i S}{ 1+ \hat v \cdot \o} + \left( \delta_{ij} - \frac{\o_i \hat{v}_j}{1+ \hat v \cdot \o}\right) T_j,
\end{split}\Ee
 while, for $\o= \o(x,y) = \frac{y-x}{|y-x|}$,
 \begin{align}
T_i &:= \p_i - \o_i \p t, \label{def:T}\\
S &:= \p_t + \hat v \cdot \nabla_x. \label{def:S}
\end{align} 
Note that 
\Be\label{T=y}
T_j f (t- |y-x| , y, v ) = \p_{y_j } [ f(t- |y-x|, y, v ) ],
\Ee
 which is a tangential derivative along the surface of a backward light cone \cite{Glassey}. On the other hand, the Vlasov equation \ref{VMfrakF} implies that 
 \Be\label{S=Lf_v}
 Sf= - \nabla_v \cdot [ (E + \hat v \times B- g \mathbf e_3)f].
 \Ee
  Therefore, in \cite{GS,Glassey}, they can take off the derivatives $T_j$, $S$ from $f$ using the integration by parts within the Green's formula of \ref{wave_eq_E}-\ref{wave_eq_B} by connecting the source terms to $f$ via \ref{charge}-\ref{current}. We refer to \cite{LS} for the recent development in the whole space case.

\section{Derivation of the Representations in a half space}
In this section we review the original Glassey-Strauss representation of $(E,B)$ in a whole space and then generalize the representation to the half space problem when the perfect conductor boundary condition \ref{PCBC}-\ref{BC:B||} of Definition \ref{def:PCBC} holds at the boundary $\p\O$. For the sake of simplicity, we may assume a single species case $\{\beta\}=\{1\}$ and $m_\beta= e_\beta=c=1$ by the renormalization.
 
 Consider the perfect conductor boundary condition of Definition \ref{def:PCBC}. We derive the representation of $E$ and $B$ satisfying the perfect conductor boundary condition at the boundary $\p\O$. We adopt convenient notations: $E= (E_\parallel, E_3) = (E_1, E_2, E_3)$, $B= (B_\parallel, B_3) = (B_1, B_2, B_3)$, $\nabla= (\nabla_\parallel, \p_3) = (\p_1,\p_2, \p_3)$, and 
\Be\label{def:bar}
\begin{split}
\bar{x}= (x_\parallel, -x_3)  \ \ \ \text{for} \ \ x=  (x_\parallel,  x_3)=  (x_1, x_2,  x_3).
\end{split}
\Ee
We refere to \cite{Guo_RV} for previous study on Vlasov equations in half space. We also refer to \cite{GSc1,GSc2,TTS} for the lower dimensional cases. 

\subsection{Tangential components of the Electronic field in a half space}\label{sec:Et}

From \ref{wave_eq_E}, \ref{PCBC}, and \ref{initialC}, we recall that, in $\O$,  
\Be \label{waveEparallel}
\begin{split}
\p_t^2 E_\parallel - \Delta_x E_\parallel = G_\parallel:=  -4 \pi \nabla_\parallel \rho - 4\pi \p_t J_\parallel   ,
\\ 
E_\parallel |_{t=0}=     E_{0\parallel}, \  \p_t E_\parallel |_{t=0} = \p_t E_{0\parallel}  ,
\end{split}
\Ee
and 
\Be\label{Dirichlet}
E_\parallel =    0 \ \   \text{ on }   \  \p \O.
\Ee

To solve the Dirichlet boundary condition, we employ the odd extension of the data: for $i=1,2$, and $x \in \R^3$,
\Be
\begin{split}\label{odd_ext}
G_i(t,x_\parallel, x_3) = &  \mathbf 1_{x_3 > 0 }  G_i(t,x ) -  \mathbf 1_{x_3 <  0 }  G_i(t,\bar x),
\\  E_{0i} ( x_\parallel, x_3) = &  \mathbf 1_{x_3 > 0 }   E_{0i} (  x) -  \mathbf 1_{x_3 <  0 }   E_{0i}( \bar x),
\\  \p_t  E_{0i} ( x_\parallel, x_3) = &  \mathbf 1_{x_3 > 0 } \p_t  E_{0i} (  x) -  \mathbf 1_{x_3 <  0 } \p_t  E_{0i}( \bar x).
\end{split}
\Ee
Then the weak solution of $E_\parallel(t,x)$ to \ref{waveEparallel} with data \ref{odd_ext} in the whole space $\R^3$ takes a form of, for $i=1,2,$ 
\begin{align}
 &E_i(t,x ) =    \frac{1}{4 \pi t^2} \int_{\p B(x;t) \cap \{ y_3 > 0 \} }  \left( t \p_t  E_{0i}( y  ) +  E_{0i}(y ) + \nabla  E_{0i} (y ) \cdot (y-x) \right) dS_y \notag
\\  &  \ \ \ \ \  + \frac{1}{4 \pi t^2} \int_{\p B(x;t) \cap \{ y_3 < 0 \} }  \big( - t \p_t  E_{0i}( \bar y  ) -  E_{0i} ( \bar y) 
 - \nabla  E_{0i} (\bar y) \cdot (\bar y - \bar x) 
\big)
 dS_y \notag
\\ & \ \ \ \  \ + \frac{1}{4 \pi } \int_{B(x;t)  \cap \{ y_3 > 0 \} }\frac{ G_i( t - |y-x|, y  ) }{|y-x| } dy \label{Eexpanmajor1}\\
&\ \  \ \ \  + \frac{1}{4 \pi } \int_{B(x;t)  \cap \{ y_3 < 0 \} }\frac{ -  G_i( t - |y-x|,  \bar y  ) }{|y-x| } dy, \label{Eexpanmajor2}
\end{align} 
where $B(x,t)= \{ y \in \R^3: |x-y| <t\}$ and $\p B(x,t)= \{ y \in \R^3: |x-y| =t\}$. Clearly the above form satisfies the zero Dirichlet boundary condition \ref{Dirichlet} at $x_3=0$ formally. From now one we regard the above form as the weak solution of \ref{waveEparallel}-\ref{Dirichlet}. The rest of task is to express \ref{Eexpanmajor1} and \ref{Eexpanmajor2}. 


\medskip 

 \textit{Expression of \ref{Eexpanmajor1}:} We follow the idea of the Glassey-Strauss (Section \ref{section_GS}). From \ref{charge}-\ref{current} and \ref{pxptST},  
\begin{align}
&\ref{Eexpanmajor1}=- \int_{B(x;t) \cap \{y_3 > 0\}}  \frac{(   \p_i \mathbf{\rho} +\p_t {J}_i )(t-|y-x|,y )}{|y-x|} \dd y
\notag\\ &\ \  =- \int_{B(x;t) \cap \{y_3 > 0\}}   \int_{\R^3} (\p_i f + \hat{v}_i \p_t f) (t-|y-x|,y ,v)  \dd v \frac{ \dd y}{|y-x|}\notag
\\ &\ \  =- \int_{B(x;t) \cap \{y_3 > 0\}}  \int_{\R^3} \frac{\o_i + \hat{v}_i}{1+ \hat{v} \cdot \o}(Sf) (t-|y-x|,y ,v) \dd v \frac{ \dd y}{|y-x|}\label{upperS}
\\ & \ \ \  \ \  - \int_{
\substack{
B(x;t) \\  \cap \{y_3 > 0\}
}
}   \int_{\R^3}  \left( \delta_{ij} - \frac{(\o_i + \hat{v}_i)\hat{v}_j}{1+\hat{v} \cdot \o} \right) T_j f (t-|y-x|,y ,v) \dd v \frac{ \dd y}{|y-x|}.  \label{upperT}
\end{align} 
Here, we followed the Einstein convention (when an index variable appears twice, it implies summation of that term over all the values of the index) and will do throughout the paper. 

For \ref{upperS}, we replace $Sf$ with \ref{S=Lf_v} and apply the integration by parts in $v$ to derive that \ref{upperS} equals
\Be \label{upperS1}
\begin{split}
  \int_{B(x;t)\cap \{y_3 > 0\} }  \int_{\R^3}  a^E_i (   v ,\o ) \cdot (E + \hat v \times B - g\mathbf e_3 )  
 f (t-|y-x|,y ,v) \dd v \frac{ \dd y}{|y-x|},
\end{split}
\Ee
where
\Be \label{aE}
 a^E_i(   v,\o ) := \nabla_v \left( \frac{\o_i + \hat{v}_i}{1+  \hat v \cdot \o } \right) =  \frac{ (e_i - \hat v_i \hat v ) ( 1 + \hat v \cdot \o ) - ( \o_i + \hat v_i ) ( \o - (\o \cdot \hat v ) \hat v ) }{\langle v \rangle ( 1 + \hat v \cdot \o )^2 }.
\Ee

For \ref{upperT}, we replace $T_j f$ with \ref{T=y} and apply the integration by parts to get \ref{upperT} equals 
\Be \label{upperT1}
\begin{split}
 &-    \int_{ \p B(x;t) \cap \{y_3>0 \}}    \o_j  \left( \delta_{ij} - \frac{(\o_i + \hat{v}_i)\hat{v}_j}{1+\hat{v} \cdot \o} \right) f(0, y,v)  \dd v \frac{\dd S_y}{|y-x|}
 \\  &+  \int_{B(x;t) \cap \{y_3= 0\}}   \int_{\R^3}  \left( \delta_{i3} - \frac{(\o_i + \hat{v}_i)\hat{v}_3}{1+\hat{v} \cdot \o} \right)  f (t-|y-x|,y_\parallel, 0 ,v)  \dd v \frac{ \dd y_\parallel}{|y-x|} \\
  &+  \int_{\substack{
B(x;t) \\  \cap \{y_3 > 0\}
}}   \int_{\R^3} 
\frac{ (|\hat{v}|^2-1  )(\hat{v}_i +  \o_i ) }{  (1+  \hat v \cdot \o )^2}
  f (t-|y-x|,y ,v)  \dd v  \frac{\dd y}{|y-x|^2}.
\end{split}
\Ee
where we have used that, from \cite{GS,Glassey},
\Be\notag
 \frac{\p}{\p y_j}\left[\frac{1}{|y-x|}\left(\delta_{ij} - \frac{(\o_i + \hat{v}_i)\hat{v}_j}{1+\hat{v} \cdot \o} \right)\right]= \frac{ (|\hat{v}|^2-1  )(\hat{v}_i +  \o_i ) }{|y-x|^2 (1+  \hat v \cdot \o )^2}. 
\Ee

\medskip 

 \textit{Expression of \ref{Eexpanmajor2}:} In order to study the expression in the lower half space we modify the idea of Glassey-Strauss slightly. Define
 \Be\label{def:o-}
 \bar \o   =  \begin{bmatrix} \o_1 & \o_2  & - \o_3 \end{bmatrix} ^T.
 \Ee
We use the same $S$ of \ref{def:S} but   
\Be\label{def:T-}
\begin{split}
 \bar{T}_3 f &=  -\p_{y_3} [f(t-|y-x|, y_\parallel, - y_3, v)] =   \p_{y_3} f - \bar{\omega}_3 \p_t f , 
\\ \bar T _i f &=  \p_{y_i} [f(t-|y-x|, y_\parallel, - y_3, v)] =   \p_{y_i} f - \bar{\omega}_i \p_t f \, \text{ for } \ i=1,2.
\end{split}
\Ee
Then we get 
\begin{align}
\p_t &=  \frac{S-   \hat{v} \cdot \bar{T} }{1+ \hat{v}  \cdot \bar{\omega} }, \label{ptST-}\\
 \p_{y_i} &=    \bar{T}_i  +   \bar{\o}_i  \frac{S-   \hat{v} \cdot \bar{T} }{1+ \hat{v}  \cdot \bar{\omega} } =  \frac{ \bar{\o}_i S }{1+ \hat{v}  \cdot \bar{\omega}   }
+ \bar{T}_i   -  \bar{\o}_i \frac{     \hat{v}   \cdot \bar{T} }{1+ \hat{v}  \cdot \bar\omega }.
 \label{pxST-}
\end{align}
Therefore, we derive
\hide\Be
\begin{split}
 \p_i + \hat{v}_i \p_t 
=& \frac{\o_i + \hat{v}_i }{1+ \hat{v}_\parallel \cdot \omega_\parallel - \hat{v}_3 \o_3}S
 + T^-_i  - (\o_i   + \hat{v}_i  )\frac{    \hat{v}_\parallel  \cdot T^-_\parallel - \hat{v}_3 T_3^-}{1+ \hat{v}_\parallel \cdot \omega_\parallel - \hat{v}_3 \o_3},
 \\
 \p_3 + \hat{v}_3 \p_t = & \frac{-\o_3 + \hat{v}_3 }{1+ \hat{v}_\parallel \cdot \omega_\parallel - \hat{v}_3 \o_3}S
- T^-_3  - (-\o_3   + \hat{v}_3  )\frac{    \hat{v}_\parallel  \cdot T^-_\parallel - \hat{v}_3 T_3^-}{1+ \hat{v}_\parallel \cdot \omega_\parallel - \hat{v}_3 \o_3}.
\end{split}
\Ee
Or
\unhide
\Be\label{ST_lower}
\begin{split}
  \p_i + \hat{v}_i \p_t 
=  \frac{ \bar{\o}_i+ \hat{v}_i }{1+
  \hat v \cdot \bar\o  
  }S
  +  \left(  \delta_{ij}  -  \frac{     \bar{\o}_i \hat{v}_j   + \hat{v}_i \hat{v}_j     }{1+ \hat v \cdot \bar\o }\right)  \bar{T}_j . 
\end{split}\Ee

Now we consider \ref{Eexpanmajor2}. From \ref{ST_lower},  
\begin{align}
&\ref{Eexpanmajor2}=  \int_{B(x;t) \cap \{y_3 <0\}}  \int_{\R^3} (\p_i f + \hat{v}_i \p_t f) (t-|y-x|,\bar y ,v) \dd v \frac{ \dd y}{|y-x|}\notag
\\& \ \ =  \int_{B(x;t) \cap \{y_3 <0\}}  \int_{\R^3} \frac{ \bar{\o}_i + \hat{v}_i }{1+ \hat v \cdot \bar{\o} }(Sf) (t-|y-x|, \bar y ,v) \dd v \frac{ \dd y}{|y-x|}\label{lowerS}
\\& \ \ +  \int_{\substack{
B(x;t) \\  \cap \{y_3 < 0\}
} }  \int_{\R^3}   \Big(  \delta_{ij} - \frac{ \bar{\o}_i \hat{v}_j+ \hat{v}_i \hat{v}_j}{1+ \hat v \cdot \bar{\o}}   \Big)  \bar{T}_j f (t-|y-x|, \bar y ,v)  \dd v \frac{ \dd y}{|y-x|}.\label{lowerT}
\end{align} 
As getting \ref{upperS1}, we derive that, with $a_i^E$ of \ref{aE}, \ref{lowerS} equals 
\Be\label{lowerS1}
 \int_{B(x;t) \cap \{y_3 < 0\}}  \int_{\R^3}  a^E_i ( v ,\bar{\o}  ) \cdot (E + \hat v \times B - g\mathbf e_3 )  
 f (t-|y-x|,\bar y ,v) \dd v \frac{ \dd y}{|y-x|}.
\Ee

 For \ref{lowerT}, applying \ref{def:T-} and the integration by parts, we derive that \ref{lowerT} equals 
\Be \label{GSlowerhalf2}
\begin{split}
 &  \int_{ \p B(x;t) \cap \{  y_3 < 0 \}} \int_{\R^3}   \bar{\o}_j    \Big(  \delta_{ij} - \frac{   \bar{\o} _i \hat{v}_j + \hat{v}_i \hat{v}_j}{1+ \hat v \cdot \bar{\o} }   \Big)  f(0, \bar y ,v)   \dd v \frac{\dd S_y}{|y-x|} 
  \\ &+  \int_{B(x;t) \cap \{y_3 =0\}}  \int_{\R^3}  \iota_3   \Big( \delta_{i3} - \frac{\bar{\o}_i  \hat{v}_3 + \hat{v}_i  \hat{v}_3}{1+ \hat v \cdot \bar{\o} }   \Big)  f(t-|y-x|,y_\parallel, 0 ,v) \dd v  \frac{ \dd y}{|y-x|}\\
  &-  \int_{B(x;t) \cap \{y_3 <0\}}  \int_{\R^3}  
   \frac{ (|\hat{v}|^2-1  )(\hat{v}_i +   \bar{\o}_i ) }{ (1+ \hat v \cdot \bar{\o} )^2} 
   f(t-|y-x|, \bar y ,v) \dd v\frac{ \dd y}{|y-x|^2}, 
\end{split}
\Ee
where we have utilized the notation
\Be\label{iota}
\iota_i = +1 \ \ \text{ for } \ i=1,2,  \ \ 
\iota_3=-1,
\Ee
and the direct computation
\Be \label{DyT}
\iota_j 
\frac{\p}{\p y_j}\left[
 \frac{1}{|y-x|}
  \Big(
  \delta_{ij} - \frac{ \iota_i \o_i \hat{v}_j + \hat{v}_i \hat{v}_j}{1+ \hat v \cdot \bar \o} 
  \Big)\right]
 =   \frac{ (|\hat{v}|^2-1  )(\hat{v}_i +   \bar{\o}_i ) }{|y-x|^2 (1+ \hat v \cdot \bar{\o} )^2}.
\Ee
\hide

Now we compute $\sum_j \iota_j \frac{\p}{\p y_j}\left[ \frac{1}{|y-x|}  \Big( \delta_{ij} - \frac{ \iota_i \o_i \hat{v}_j + \hat{v}_i \hat{v}_j}{1+ \hat v \cdot \o^-}   \Big)\right] $. Note that 
\[
\p_{y_j} |y-x|= \frac{(y-x)_j}{|y-x|}, \ \ \ 
\p_{y_j} \o_i = \frac{1}{|y-x|} \left(
\delta_{ij} - \frac{(y-x)_i (y-x)_j}{|y-x|^2}
\right)
 \]
We have

where we have computed 
\Be
\begin{split}
 \sum_{j} \bigg\{& - \iota_j |y-x|^{-2}  \o_j  \Big(
  \delta_{ij} - \frac{ \iota_i  \o_i \hat{v}_j + \hat{v}_i \hat{v}_j}{1+ \hat v \cdot \o^-} 
  \Big)\\
 &+ \iota_j  |y-x|^{-1} (-1)^2 [1+ \hat v \cdot \o^- ]^{-2} 
  \left(\sum_k
 \hat{v}_k \iota_k \p_{y_j} \o_k
 \right)
 \left( \iota_i \o_i \hat{v}_j + \hat{v}_i \hat{v}_j \right)\\
 &+ \iota_j |y-x|^{-1} (-1) [1+ \hat v \cdot \o^-]^{-1}
 \iota_i \p_{y_j} \o_i \hat{v}_j \bigg\}\\
 \end{split}
\Ee
\Be
\begin{split}
 = \frac{1}{|y-x|^2 [1+ \hat v \cdot \o^-]^2}\sum_{j} \bigg\{&
 -[1+ \hat v \cdot \o^-]^2 \delta_{ij} \iota_j \o_j
 + [1+ \hat v \cdot \o^-] \{ \iota_i \o_i \iota_j\hat{v}_j + \hat{v}_i \iota_j\hat{v}_j\}   \o_j \\
 &+ \{ \iota_i \o_i  \iota_j\hat{v}_j + \hat{v}_i \iota_j \hat{v}_j\}
 \sum_k \iota_k \hat{v}_k (\delta_{kj} - \o_k \o_j) \\
 &- [1+ \hat v \cdot \o^-] \iota_j \hat{v}_j ( \iota_i\delta_{ij} - \iota_i \o_i \o_j )
 \bigg\}\\
\end{split}
\Ee
Now the summation equals
\Be
\begin{split}
& (\hat v \cdot \o^- ) ^2\hat{v}_i  - (\hat v \cdot \o^-)  ^2\hat{v}_i  + |\hat{v}|^2 \hat{v}_i\\
&-  (\hat v \cdot \o^-)^2 \iota_i \o_i + (\hat v \cdot \o^-)^2 \iota_i \o_i
+ (\hat v \cdot \o^-) \hat{v}_i 
+|\hat{v}|^2 \iota_i \o_i -(\hat v \cdot \o^-)^2 \iota_i \o_i 
- (\hat v \cdot \o^-) \hat{v}_i  + (\hat v \cdot \o^-)^2 \iota_i \o_i \\
&- 2(\hat v \cdot \o^-) \iota_i \o_i +(\hat v \cdot \o^-) \iota_i \o_i 
- (\iota_i)^2 \hat{v}_i  + (\hat v \cdot \o^-) \iota_i \o_i \\
&- \iota_i \o_i \\
=&  |\hat{v}|^2 \hat{v}_i  + |\hat{v}|^2 \iota_i \o_i- (\iota_i)^2 \hat{v}_i - \iota_i \o_i .
\end{split}
\Ee

\unhide

\subsection{Normal components of the Electronic field in a half space}
From \ref{wave_eq_E}, \ref{BC:E_3} and \ref{initialC}, we have 
\Be \label{waveE3}
\begin{split}
\p_t^2 E_3 - \Delta_x E_3 = G_3:=   -4 \pi \p_3 \rho - 4\pi \p_t J_3  , 
\\ E_3 |_{t=0}=     E_{03}, \  \p_t E_3 |_{t=0} = \p_t E_{03}    ,
\end{split}
\Ee
and 
\Be 
  \p_3 E_3 =    4 \pi \rho   \ \   \text{ on } \ \  \p \O. \label{Neumann}
\Ee
It is convenient to decompose the solution into two parts: one with the Neumann boundary condition of \ref{waveE3} and the zero forcing term and initial data
\Be \label{waveE3b}
\begin{split}
\p_t^2 w  - \Delta_x w  =  0     \ \  &\text{ in } \ \  \O,
\\ w  |_{t=0} =  0, \ \p_t w |_{t=0}  = 0   \ \  &\text{ in } \ \  \O,
\\ \p_3 w= 4 \pi \rho   \ \  &\text{ on } \ \  \p \O,
\end{split}\Ee
and the other part $\tilde E_3$ with the initial data of \ref{waveE3} and the zero Neumann boundary condition. We achieve it by the even extension trick. Recall $\bar x$ in \ref{def:bar}. For $x \in \R^3$, define \Be \label{wavetildeE32}
\begin{split}
G_3(t,x) = & \mathbf{1}_{x_3 > 0 } G_3(t,x ) +   \mathbf{1}_{x_3 < 0 } G_3(t,\bar x )  , 
\\  E_{03} (x)  = &  \mathbf 1_{x_3 > 0 }    E_{03  } (x ) +  \mathbf 1_{x_3 <  0 }    E_{03 } ( \bar x ),
\\  \p_t    E_{03} (x)   = &  \mathbf 1_{x_3 > 0 } \p_t  E_{03} ( x ) +  \mathbf 1_{x_3 <  0 } \p_t  E_{03}( \bar x ).
\end{split}
\Ee
The weak solution $\tilde E_3$ to \ref{waveE3} with the data \ref{wavetildeE32} in the whole space $\R^3$ take a form of  
\begin{align}
\tilde E_3(t,x ) &=    \frac{1}{4 \pi t^2} 
\int_{
\substack{\p B(x;t)\\ \cap \{ y_3 > 0 \}} }  \left( t \p_t  E_{03}( y ) +  E_{03}(y ) + \nabla  E_{03} (y ) \cdot (y-x) \right) dS_y\notag
\\  &   + \frac{1}{4 \pi t^2} \int_{ \substack{\p B(x;t) \\ \cap \{ y_3 < 0 \} }}  \big(  t \p_t  E_{03}( \bar y  ) +  E_{03} (\bar y) 
  + \nabla  E_{03} ( \bar y ) \cdot (\bar y - \bar x  )   \big) dS_y\notag
\\ &  + \frac{1}{4 \pi } \int_{ \substack{ B(x;t) \\ \cap \{ y_3 > 0 \}} }\frac{ G_3( t - |y-x|, y  ) }{|y-x| } dy \label{tildeE3formula1}\\
&   + \frac{1}{4 \pi } \int_{\substack{ B(x;t)  \\  \cap \{ y_3 < 0 \} }}\frac{  G_3( t - |y-x|, \bar y ) }{|y-x| } dy. \label{tildeE3formula2}
\end{align} 
Following the same argument to expand \ref{Eexpanmajor1} into \ref{upperS}-\ref{upperT1} and \ref{Eexpanmajor2} into \ref{lowerS}-\ref{GSlowerhalf2}, we derive that 
\Be \label{E3expression1}
\begin{split}
&\ref{tildeE3formula1}\\
& =  \int_{\substack{ B(x;t)  \\ \cap \{y_3 > 0\}}}  \int_{\R^3} a^E_3 ( v ,\o ) \cdot (E + \hat v \times B - g\mathbf e_3 )  
 f (t-|y-x|,y ,v) \dd v \frac{ \dd y}{|y-x|} 
 \\  & + \int_{B(x;t) \cap \{y_3 > 0\}}   \int_{\R^3}\frac{ (|\hat{v}|^2-1  )(\hat{v}_3 +  \o_3 ) }{ (1+  \hat v \cdot \o )^2} f (t-|y-x|,y ,v)  \dd v  \frac{ \dd y}{|y-x|^2}
 \\ & -  \int_{ \p B(x;t) \cap \{  y_3>0 \}}   \o_j  \left( \delta_{3j} - \frac{(\o_3 + \hat{v}_3)\hat{v}_j}{1+\hat{v} \cdot \o} \right) f(0, y,v)  \dd v \frac{\dd S_y}{|y-x|}
 \\ & + \int_{B(x;t) \cap \{y_3= 0\}}   \int_{\R^3}  \left(  1 - \frac{(\o_3 + \hat{v}_3)\hat{v}_3}{1+\hat{v} \cdot \o} \right)  f (t-|y-x|,y_\parallel, 0 ,v)  \dd v \frac{ \dd y_\parallel}{|y-x|},
\end{split}
\Ee 
\Be
\begin{split}\label{E3expression2}
&  \ref{tildeE3formula2}\\ 
&=   - \int_{ \substack{B(x;t) \\  \cap \{y_3 < 0\}}}  \int_{\R^3}  a^E_3 ( v ,\bar{\o}  ) \cdot (E + \hat v \times B - g\mathbf e_3 )  
 f (t-|y-x|, \bar y ,v) \dd v \frac{ \dd y}{|y-x|}
\\
&+ \int_{B(x;t) \cap \{y_3 <0\}}  \int_{\R^3}  
   \frac{ (|\hat{v}|^2-1  )(\hat{v}_3 +   \bar{\o}_3 ) }{ (1+ \hat v \cdot \bar{\o} )^2} 
   f(t-|y-x|,\bar y  ,v) \dd v\frac{ \dd y}{|y-x|^2} \\
&- \int_{\p B(x;t) \cap \{  y_3 < 0 \}} \int_{\R^3}   \bar{\o}_j    \Big(  \delta_{3j} - \frac{   \bar{\o} _3 \hat{v}_j + \hat{v}_3 \hat{v}_j}{1+ \hat v \cdot \bar{\o} }   \Big)  f(0,\bar y ,v)   \dd v \frac{\dd S_y}{|y-x|} 
  \\ &+  \int_{B(x;t) \cap \{y_3 =0\}}  \int_{\R^3}     \Big( 1- \frac{ (\bar{\o}_3   + \hat{v}_3 ) \hat{v}_3}{1+ \hat v \cdot \bar{\o} }   \Big)  f(t-|y-x|,y_\parallel, 0 ,v) \dd v  \frac{ \dd y}{|y-x|}.
\end{split}
\Ee
Note that the weak derivative $\p_3$ to the form of $\tilde{E}_3$ solves the linear wave equation \ref{waveE3} with oddly extended forcing term and the initial data in the sense of distributions. From the argument of Section \ref{sec:Et}, we conclude that 
\Be\label{p3tildeE3=0}
\p_3  \tilde E_3=0 \ \   \text{ on } \ \  \p \O.
\Ee

\subsubsection{Wave equation with the Neumann boundary condition} Now we consider \ref{waveE3b}. We assume $\rho(t,x)$ for all $t\leq 0$, which implies $w(t,x)=0$ for all $t\leq 0$. Define the Laplace transformation:
\Be
W(p, x) = \int^\infty_{- \infty} e^{-p t} w (t,x ) dt, \ \  \  \ 
R(p, x) = \int^\infty_{- \infty} e^{-p t} \rho (t,x ) dt
.\label{LaplaceT}
\Ee
Then $W$ solves the Helmholtz equation with the same Neumann boundary condition:
\Be \label{Helmholtz}
\begin{split}
p^2  W - \Delta_x   W  =  0 & \ \  \text{ in }   \  \O,
\\ \p_n  W =  4 \pi R    & \ \  \text{ on }   \  \p \O.
\end{split}
\Ee
The solution for $( p^2 - \Delta _x  ) \Phi(x) =   \delta(x) $ in $\mathbb R^3$ is known as $  \frac{1}{4\pi } \frac{e^{\pm p |x| }}{ |x| }$. We choose \Be \label{fundaHelmPhi}
 \Phi(x) =  \frac{1}{4\pi } \frac{e^{- p |x| }}{ |x| }.
\Ee
We have the following identities:
\begin{lemma}
Suppose $u \in C^2(\bar \O )$ is an arbitrary function. For a fixed $x \in \O$ and $\Phi$ in \ref{fundaHelmPhi}, we have
\Be \label{uPhiintrep}
\begin{split}
u(x) &=   \int_\O \Phi(y-x) (   p^2- \Delta_x ) u (y) dy \\
&+ \int_{\p \O } \left[  \Phi(y-x) \p_n u(y) - u(y) \p_n \Phi(y-x) \right] dS_y.\end{split}
\Ee
\end{lemma}
\begin{proof}The proof is rather standard. Fix $x \in \O$. Let $0 < \e \ll 1 $, and $B(x,\e)$ be a ball centered at $x$ with radius $\e$ such that $B(x, \e ) \subset \O$. Let $V_\e = \O - B(x,\e)$. Then, by the integration by parts,
\Be \notag
\begin{split}
& - \int_{V_\e} \Phi(y-x ) ( \Delta_y - p^2 ) u(y) dy 
\\ 
&=    \int_{ \p \O }  u(y)   \p_n \Phi(y-x ) dS_y +  \int_{\p B(x,\e) }  u(y)   \p_n \Phi(y-x ) dS_y 
\\ & \ \ -  \int_{\p \O}    \Phi(y-x )   \p_n  u(y)dS_y - \int_{\p B(x, \e) }    \Phi(y-x )   \p_n  u(y)dS_y.
\end{split}
\Ee
From \ref{fundaHelmPhi}, $
 \int_{\p B(x, \e) }    \Phi(y-x )   \p_n  u(y)dS_y \lesssim 4\pi \e^2 \frac{e^{|p| \e }}{4 \pi \e }  \to 0,$ as $\e \to 0.$ 
And by direct computation, 
\Be\notag
\begin{split}
 &\int_{\p B(x,\e) }  u(y)   \p_n \Phi(y-x ) dS_y =     \int_{\p B(x,\e) }  u(y)  \frac{-(y-x)}{|y-x| } \cdot  \nabla \Phi(y-x ) dS_y 
 \\ = & \frac{1}{4 \pi }  \int_{\p B(x,\e) }  u(y)  \frac{-(y-x)}{|y-x| } \cdot  ( - i p |y-x | - 1 ) \frac{ e^{- i p |y-x | } (y-x) }{|y-x |^3 } dS_y 
 \\ = &  \frac{1}{4 \pi }  \int_{\p B(x,\e) } \left( - (- p |y-x| - 1 ) \frac{ e^{- p|y-x |  }}{|y-x |^2 } \right) u(y)   dS_y
  \\ = &   \left( ( 1 - (- p) \e ) e^{- p \e} \right)  \left(  \frac{1}{4 \pi \e^2  }   \int_{\p B(x,\e) } u(y)   dS_y \right)
  \to   u(x), \text{ as } \e \to 0.
\end{split}
\Ee
Combining all together, and letting $\e \to 0$ we get \ref{uPhiintrep}.\end{proof}

Next for $x \in \O$, let $\phi^x(y)$ be the function such that 
\Be \label{phixyforO}
\begin{split}
(\Delta_y  - p^2 ) \phi_N^x(y) =  0 \ \  &\text{ in }  \ \ \O,
\\ \p_n \phi_N^x(y) =   \p_n \Phi(y-x ) \ \  &\text{ on } \ \  \p \O.
\end{split}
\Ee 
The integration by parts implies
\Be \label{phixyintbyparts}
\begin{split}
0 = & \int_{\O} ( \Delta_y - p^2 ) \phi_N^x(y) u(y) dy 
\\ = & \int_{\O } ( \Delta_y - p^2 )u(y)  \phi_N^x(y) dy + \int_{\p \O } [ \p_n \phi_N^x(y) u(y)  - \phi_N^x(y) \p_n u(y) ] dS_y
\\ = & \int_{\O } ( \Delta_y - p^2 )u(y)  \phi_N^x(y) dy + \int_{\p \O } [ \p_n  \Phi(y-x ) u(y)  - \phi_N^x(y) \p_n u(y) ] dS_y.
\end{split}
\Ee
By adding \ref{phixyintbyparts} to \ref{uPhiintrep}, we derive that 
\Be \label{urepPhi2}
\begin{split}
u(x) = &  - \int_{\O} \left( \Phi(y-x) -\phi_N^x(y)     \right) ( \Delta_y - p^2 )u(y) dy\\
& + \int_{\p \O } ( \Phi(y-x) -\phi_N^x(y)  ) \p_n u(y) dS_y.
\end{split}
\Ee
For the half space $\O = \mathbb R_+^3$, we have, with $\bar x$ in \ref{def:bar},  
\Be \label{phiNxyformhalf}
\phi_N^x(y) = - \Phi(y- \bar x ) .
\Ee
\hide
Since $x,y \in  \mathbb R_+^3$, so $y - \tilde x \neq 0$, and $(\Delta_y  - p^2 ) ( - \Phi(y- \tilde x )) = 0$ for all $y \in  \mathbb R_+^3$. Also, by direct computation
\[
\p_{y_3}   \Phi(y-  x ) =  \frac{1}{4 \pi } \left(  \frac{(\pm p |y-x|  -1 ) e^{\pm p |y-x | } (y_3 - x_3 )  }{ |y-x |^3 }  \right).
\]
For $y \in \p \O$, $y_3 = 0$ and $|y - x | = |y - \tilde x | $, thus
\Be
\begin{split}
\p_{y_3} \phi_N^x(y) |_{y_3 = 0 }  = - \p_{y_3} \Phi(y- \tilde x )   |_{y_3 = 0 }  = & - \frac{1}{4 \pi } \left(  \frac{(\pm p |y- \tilde x|  -1 ) e^{\pm p |y-\tilde x | } (y_3 + x_3 )  }{ |y-\tilde x |^3 }  \right)  |_{y_3 = 0 } 
\\ = &  - \frac{1}{4 \pi } \left(  \frac{(\pm p |y-  x|  -1 ) e^{\pm p |y- x | }  x_3   }{ |y- x |^3 }  \right) 
\\ = &  \p_{y_3}   \Phi(y-  x )  |_{y_3 = 0 },
\end{split} 
\Ee
and we proved \ref{phiNxyformhalf}. \unhide

Finally, we derive that, from \ref{urepPhi2} and \ref{phiNxyformhalf}:
\begin{lemma} For $\O = \R^3_+$, and $\Phi$ in \ref{fundaHelmPhi}, 
\Be
\begin{split}\label{formula:u}
u(x) =& - \int_{\O} \left(  \Phi(y-x) + \Phi(y-\bar x)  \right) ( \Delta_y - p^2 )u(y) dy\\
& + \int_{\p \O } ( \Phi(y-x) + \Phi(y-\bar x)  ) \p_n u(y) dS_y.
\end{split}
\Ee
\end{lemma}

By applying \ref{phiNxyformhalf} to \ref{Helmholtz}, we derive that  
\Be
\begin{split}\label{form_W}
W(p,x) &=   \int_{\p \O } ( \Phi(y-x) + \Phi(y-\bar x)  ) 4 \pi R (y) d S_y
\\  &=  2 \int_{\mathbb R^2 }\frac{e^{-p ( (y_1 - x_1)^2 +(y_2 - x_2 )^2 + x_3^2  )^{1/2}}}{ (y_1 - x_1)^2 +((y_2 - x_2 )^2 + x_3^2  )^{1/2}} R(y_1, y_2 ) d y_1 dy_2 .
\end{split}
\Ee
Using the inverse Laplace transform, we derive that 
%
%
\Be \label{wexpression1}
\begin{split}
&w(t,x) =   \frac{1}{2\pi } \int_{-\infty}^\infty e^{(p_1 + i p_2 )t} W( p_1 + i p_2 ,x ) dp_2\\
& \ \ =    \frac{1}{ \pi   } \int_{-\infty}^\infty d p_2  \  e^{(p_1 + i p_2 )t}    \int_{\mathbb R^2 }d y_1 dy_2\frac{e^{ -( p_1 + i p_2) ( (y_1 - x_1)^2 +(y_2 - x_2 )^2 + x_3^2  )^{1/2}}}{ (y_1 - x_1)^2 +((y_2 - x_2 )^2 + x_3^2  )^{1/2}} \\
& \ \ \ \ \ \       \times    \int_{-\infty}^{\infty}  ds  \   e^{-( p_1 + i p_2) s } (- \rho(s,y_1,y_2 ))     .
\end{split}
\Ee
Finally, we derive that, using the identity $\int_{-\infty}^\infty e^{i p_2 t } d p_2 = 2 \pi \delta(t) $, 
\Be \label{wexpression2}
\begin{split}
 &w(t,x)   = \frac{- 1}{ \pi   } \int_{\mathbb R^2 } \int_\R \int_\R \frac{e^{( p_1 + i p_2 )(t -s -  \sqrt{| y_\parallel - x_\parallel |^2 + x_3^2 } )  }}{ \sqrt{| y_\parallel - x_\parallel |^2 + x_3^2 } }  \rho( s, y_\parallel ) dp_2  ds dy_\parallel 
\\  &= -  2  \int_{\mathbb R^2 } \int_\R  \frac{e^{p_1  ( t-s -  \sqrt{| y_\parallel - x_\parallel |^2 + x_3^2 } ) }  \delta(t-s -  \sqrt{| y_\parallel - x_\parallel |^2 + x_3^2 } ) }{ \sqrt{| y_\parallel - x_\parallel |^2 + x_3^2 } } \rho ( s, y_\parallel ) ds dy_\parallel
\\ &=  - 2  \int_{
\sqrt{| y_\parallel - x_\parallel |^2 + x_3^2 }<t
}  \frac{\rho (t -  \sqrt{| y_\parallel - x_\parallel |^2 + x_3^2 },y_\parallel ) }{ \sqrt{| y_\parallel - x_\parallel |^2 + x_3^2 }}   dy_\parallel.
\end{split}
\Ee

\subsection{Summary}

Collecting the terms, we conclude the following formula. 
\begin{proposition}
 
 \Be    \begin{split}
 & E_i(t,x)  
 =   \frac{1}{ 4 \pi  t^2} \int_{\p B(x;t) \cap \{ y_3 > 0 \} }  \left( t \p_t  E_{0i}( y) +  E_{0i} (y) + \nabla  E_{0i} (y) \cdot (y-x) \right) dS_y
\\   & + \frac{1}{ 4 \pi  t^2} \int_{\p B(x;t) \cap \{ y_3 < 0 \} } \iota_i \big(- t \p_t  E_{0i}( \bar y ) -  E_{0i} (\bar y)   - \nabla  E_{0i} (\bar y)  \cdot (\bar y - \bar x ) \big) dS_y 
 \\  
   &  +   \int_{B(x;t) \cap \{y_3 > 0\}}  \int_{\R^3}\frac{ (|\hat{v}|^2-1  )(\hat{v}_i +  \o_i ) }{|y-x|^2 (1+  \hat v \cdot \o )^2} f(t-|y-x|,y ,v)\dd v \dd y
 \\ 
  &- \int_{B(x;t) \cap \{y_3 <0\}}   \int_{\R^3} \iota_i \frac{ (|\hat{v}|^2-1 )(\hat{v}_i +  \bar \o_i ) }{|y-x|^2 (1+ \hat v \cdot \bar \o )^2}f(t-|y-x|,\bar y ,v) \dd v \dd y
 \\ &  
  +  \int_{B(x;t) \cap \{y_3 > 0\}}  \int_{\R^3}    a^E_i(   v,\o ) \cdot (E + \hat v \times B - g\mathbf e_3 )   f (t-|y-x|,y ,v) \dd v \frac{ \dd y}{|y-x|}
 \\ & 
  -   \int_{B(x;t) \cap \{y_3 <  0\}}  \int_{\R^3} \iota_i   a^E_i ( v ,\bar \o ) \cdot (E + \hat v \times B - g\mathbf e_3 )   f (t-|y-x|, \bar y ,v) \dd v \frac{ \dd y}{|y-x|}
\\ 
&  + \int_{B(x;t) \cap \{y_3= 0\}} \int_{\R^3} \left( \delta_{i 3 }  -  \frac{(\o_i + \hat{v}_i)\hat{v}_3}{1+ \hat v \cdot \o }  \right) f (t-|y-x|,y_\parallel, 0 ,v) \dd v\frac{ \dd y_\parallel}{|y-x|}
\\  
& -  \int_{B(x;t) \cap \{y_3 =0\}} \int_{\R^3} \iota_i \left( \delta_{i 3 } -   \frac{ \bar \o_i \hat{v}_3 + \hat{v}_i  \hat{v}_3}{1+ \hat v \cdot \bar \o} \right) f(t-|y-x|,y_\parallel, 0 ,v) \dd v \frac{ \dd y_\parallel }{|y-x|}
\\ 
&-  \int_{\p B(x;t) \cap \{ y_3 > 0 \} }  \int_{\mathbb R^3}   \o_j  \left(\delta_{ij} - \frac{(\o_i + \hat{v}_i)\hat{v}_j}{1+ \hat v \cdot \o }\right) f(0,y,v) \dd v \frac{\dd S_y}{|y-x|}
 \\ 
 &  + \int_{\p B(x;t) \cap \{ y_3 < 0 \} }  \int_{\R^3}   \iota_i  \bar \o_j   \Big( \delta_{ij} - \frac{ \bar \o_i \hat{v}_j + \hat{v}_i \hat{v}_j}{1+ \hat v \cdot \bar \o}   \Big)  f(0, \bar y ,v)  \dd v \frac{\dd S_y}{|y-x|}
 \\ 
 & - \delta_{i3}   \int_{ B(x;t) \cap \{y_3 =0\}}  \int_{\mathbb R^3 }  \frac{ 2 f (t - |y-x |, y_\parallel, 0, v )}{|y-x | }  dv d  S_y .
\end{split}
\Ee

 \end{proposition}

\subsection{Representation of the Magnetic field in a half space}

Next, we solve for $B$. For $B_1, B_2$ we have, for $i =1,2,$
\Be \label{wavetildeB}
 \begin{split}
\p_t^2  B_i - \Delta_x  B_i = & 4 \pi ( \nabla_x \times J)_i:= H_i  \ \  \text{ in }  \ \ \O,  
\\ \p_{x_3} B_1 = &  4 \pi J_2  , \ \p_{x_3} B_2 =  4 \pi J_1 \ \  \text{ on } \ \  \p \O,
\\  B_i(0,x) = & B_{0 i}, \p_t  B_i(0,x) = \p_t B_{0 i} \ \  \text{ in }  \ \ \O.
\end{split} \Ee

To solve \ref{wavetildeB} we write $B_i = \tilde B_i + B_{bi} $ with $\tilde B_i$ satisfies the wave equation in $(0, \infty) \times \mathbb R^3 $ with even extension in $x_3$:
\Be \label{wavetildeBi}
\begin{split}
\p_t^2 \tilde B_i - \Delta_x \tilde B_i = &  \mathbf 1_{x_3 > 0 }  H_i(t,x) +  \mathbf 1_{x_3 <  0 }  H_i(t,\bar x),
\\  \tilde B_i(0,x)  = &  \mathbf 1_{x_3 > 0 }   B_{0i} (x) +  \mathbf 1_{x_3 <  0 }   B_{0i} (\bar x),
\\  \p_t  \tilde B_i (0 , x) = &  \mathbf 1_{x_3 > 0 } \p_t  B_{0i} (x) +  \mathbf 1_{x_3 <  0 } \p_t  B_{0i}(\bar x).
\end{split}
\Ee
And $B_{bi}$ satisfies
\Be \label{Bibeq}
\begin{split}
\p_t^2 B_{bi} - \Delta_x B_{bi} = 0 & \text{ in } \O,
\\ B_{bi} (0,x) = 0, \p_t B_{bi} = 0 & \text{ in } \O,
\\ \p_{x_3} B_{b1} = 4 \pi J_2, \ \p_{x_3} B_{b2} = - 4 \pi J_1 & \text{ on } \O.
\end{split}
\Ee
Then from \ref{wavetildeBi},
\[ \label{tildeE3formula}
\begin{split}
& \tilde B_i (t,x )
\\ =  &  \frac{1}{4 \pi t^2} \int_{\p B(x;t) \cap \{ y_3 > 0 \} }  \left( t \p_t  B_{0i} (y ) +   B_{0i}(y) + \nabla   B_{0i} (y) \cdot (y-x) \right) dS_y
\\  & + \frac{1}{4 \pi t^2} \int_{\p B(x;t) \cap \{ y_3 < 0 \} }  \big(  t \p_t   B_{0i}(\bar y ) +   B_{0i}(\bar y)   + \nabla   B_{0i} (\bar y) \cdot (\bar y - \bar x)   \big) dS_y
\\ & + \frac{1}{4 \pi } \int_{B(x;t)  \cap \{ y_3 > 0 \} }\frac{ H_i ( t - |y-x|, y ) }{|y-x| } dy + \frac{1}{4 \pi } \int_{B(x;t)  \cap \{ y_3 < 0 \} }\frac{  H_i( t - |y-x|, \bar y ) }{|y-x| } dy.
\end{split}
\]
Applying \ref{wexpression2} to \ref{Bibeq},
\Be
\begin{split}
B_{bi}(t,x) = (-1)^i 2  \int_{ B(x;t) \cap \{y_3 =0\}}  \int_{\mathbb R^3 }  \frac{ \hat v_{\underline{i} }  f (t - |y-x |, y_\parallel, 0, v )}{|y-x | }  dv d  S_y,
\end{split}
\Ee
where we define
\[
\underline i = \begin{cases} 2, \text{ if } i=1, \\ 1, \text{ if } i=2. \end{cases}
\]
Thus,
\Be \label{Bi12rep}
\begin{split}
&B_i(t,x) \\
= &  \frac{1}{4 \pi t^2} \int_{\p B(x;t) \cap \{ y_3 > 0 \} }  \left( t \p_t  B_{0i} ( y ) +   B_{0i}(y) + \nabla   B_{0i} (y) \cdot (y-x) \right) dS_y
\\  & + \frac{1}{4 \pi t^2} \int_{\p B(x;t) \cap \{ y_3 < 0 \} }  \big(  t \p_t   B_{0i}(\bar y ) +   B_{0i}(\bar y)  + \nabla   B_{0i} (\bar y) \cdot (\bar y - \bar x)  \big) dS_y
\\ & + \frac{1}{4 \pi } \int_{B(x;t)  \cap \{ y_3 > 0 \} }\frac{ H_i ( t - |y-x|, y ) }{|y-x| } dy + \frac{1}{4 \pi } \int_{B(x;t)  \cap \{ y_3 < 0 \} }\frac{  H_i( t - |y-x|, \bar y ) }{|y-x| } dy
\\ & + (-1)^i 2  \int_{ B(x;t) \cap \{y_3 =0\}}  \int_{\mathbb R^3 }  \frac{ \hat v_{\underline i }  f (t - |y-x |, y_\parallel, 0, v )}{|y-x | }  dv d  S_y.
\end{split}
\Ee

On the other hand, $B_3(t,x)$ satisfies
\[
\begin{split}
\p_t^2 B_3 - \Delta_x B_3 = & 4 \pi (\nabla_x \times j )_3 := H_3 \text{ in }  \O,
\\ B_3(0,x) = & B_{03}, \p_t B_3(0,x) = \p_t B_{03} \text{ in } \O,
\\ B_3 = & 0 \text{ on }  \p \O.
\end{split}
\]
Using the odd extension in $x_3$:
\[
\begin{split}
H_3(t,x)  = &  \mathbf 1_{x_3 > 0 }  H_3(t,x) -  \mathbf 1_{x_3 <  0 }  H_3(t,\bar x),
\\   B_{03} (x)  = &  \mathbf 1_{x_3 > 0 }   B_{03} (x) -  \mathbf 1_{x_3 <  0 }   B_{03} (\bar x),
\\  \p_t   B_{03} (0 , x) = &  \mathbf 1_{x_3 > 0 } \p_t  B_{03} (x) -  \mathbf 1_{x_3 <  0 } \p_t  B_{03}(\bar x),
\end{split}
\]
we get the expression for $B_3$:
\Be \label{B3rep}
\begin{split}
&B_3(t,x) \\
= &  \frac{1}{4 \pi t^2} \int_{\p B(x;t) \cap \{ y_3 > 0 \} }  \left( t \p_t  B_{03} (y ) +   B_{30}(y) + \nabla   B_{03}(y) \cdot (y-x) \right) dS_y
\\  & - \frac{1}{4 \pi t^2} \int_{\p B(x;t) \cap \{ y_3 < 0 \} }  \big(  t \p_t   B_{03}( \bar y ) +   B_{03} (\bar y)  + \nabla   B_{03} (\bar y) \cdot (\bar y - \bar x )  \big) dS_y
\\ & + \frac{1}{4 \pi } \int_{B(x;t)  \cap \{ y_3 > 0 \} }\frac{ H_3 ( t - |y-x|, y ) }{|y-x| } dy - \frac{1}{4 \pi } \int_{B(x;t)  \cap \{ y_3 < 0 \} }\frac{  H_3( t - |y-x|, \bar y ) }{|y-x| } dy.
\end{split}
\Ee
Combining \ref{Bi12rep} and \ref{B3rep}, we get for $i=1,2,3$,
\begin{align}
  B_i(t,x) =
 \notag &  \frac{1}{4 \pi t^2} \int_{\p B(x;t) \cap \{ y_3 > 0 \} }  \left( t \p_t  B_{0i} ( y ) +   B_{0i}(y) + \nabla   B_{0i} (y) \cdot (y-x) \right) dS_y
\\   \notag & + \frac{\iota_i}{4 \pi t^2} \int_{\p B(x;t) \cap \{ y_3 < 0 \} }  \big(  t \p_t   B_{0i}( \bar y ) +   B_{0i}(\bar y)   + \nabla  B_{0i} (\bar y) \cdot (\bar y - \bar x )   \big) dS_y
\\ \label{Bexpanmajor1} & + \frac{1}{4 \pi } \int_{B(x;t)  \cap \{ y_3 > 0 \} }\frac{ H_i ( t - |y-x|, y ) }{|y-x| } dy 
\\  \label{Bexpanmajor2} & + \frac{\iota_i}{4 \pi } \int_{B(x;t)  \cap \{ y_3 < 0 \} }\frac{  H_i( t - |y-x|, \bar y ) }{|y-x| } dy
\\ \notag & + (-1)^i  2( 1- \delta_{i3} )  \int_{ B(x;t) \cap \{y_3 =0\}}  \int_{\mathbb R^3 }  \frac{ \hat v_{\underline i }  f (t - |y-x |, y_\parallel, 0, v )}{|y-x | }  dv d  S_y.
\end{align}


Using \ref{pxptST}, we have
\begin{align}
\notag \ref{Bexpanmajor1}  =  &    \int_{B(x;t)   \cap \{ y_3 > 0 \} } \int_{\mathbb R^3} \frac{(  \nabla_x f \times \hat v )_i ( t - |y-x|, y, v ) }{|y-x| } dy
\\  \label{BupperS} = &   \int_{B(x;t)  \cap \{ y_3 > 0 \} } \int_{\mathbb R^3}   \frac{ (\o \times \hat v)_i  }{ 1+ \hat{v} \cdot \o}  S  f ( t - |y-x|, y ,v  ) dv \frac{ dy} {|y-x|}
\\ \label{BupperT} & + \int_{B(x;t)  \cap \{ y_3 > 0 \} } \int_{\mathbb R^3}   \left(  (T \times \hat v)_i  - \frac{ (\o \times \hat v )_i \hat v \cdot T }{1+ \hat{v} \cdot  \o}\right)    f ( t - |y-x|, y, v ) dv \frac{ dy} {|y-x|}.
\end{align}

For \ref{BupperS}, we replace $Sf$ with \ref{S=Lf_v} and apply the integration by parts in $v$ to derive that \ref{BupperS} equals
\Be \label{BupperS1}
\begin{split}
\int_{B(x;t) \cap \{y_3 > 0\}}  \int_{\R^3}   a^B_i ( v ,\o ) \cdot (E + \hat v \times B - g\mathbf e_3 )  
 f (t-|y-x|,y ,v) \dd v \frac{ \dd y}{|y-x|},
\end{split}
\Ee
where
\Be \label{aB}
 a^B_i(   v,\o ) := \nabla_v \left( \frac{( \o \times  \hat v )_i }{ 1 + \hat v \cdot \o } \right)  =   \frac{ \nabla_v [ (\o \times v)_i ] }{\sqrt{1+|v|^2 } (1 + \hat  v \cdot \o  ) }  + \frac{ (\o \times v)_i ( \hat v  + \o ) }{ ( \sqrt{1+|v|^2 } (1  + \hat v \cdot \o ) )^2 }.
\Ee

For \ref{BupperT}, we replace $T_j f$ with \ref{T=y} and apply the integration by parts to get \ref{BupperT} equals 
\Be \label{BupperT1}
\begin{split}
 & \int_{  \p B(x; t)  \cap \{ y_3 > 0 \} } \int_{\mathbb R^3}   ( \o \times \hat v)_i  \left( 1   - \frac{  \hat v \cdot \o  }{1+ \hat{v} \cdot  \o}\right)     f ( 0, y, v ) dv \frac{ dS_y} {t}
\\ & +   \int_{ B(x;t)  \cap \{ y_3 = 0 \} } \int_{\mathbb R^3}      \left( - (e_3 \times \hat v)_i   + \frac{ ( \o \times \hat v  )_i }{1+ \hat{v} \cdot  \o}  ( \hat v \cdot e_3 )  \right)     f ( t- |y-x|, y_\parallel, 0, v ) dv \frac{ dy_\parallel} {|y-x|}
\\ & + \int_{B(x;t)  \cap \{ y_3 > 0 \} }  \int_{\mathbb R^3}    \frac{   ( \o \times \hat v  )_i \left( 1   -  |\hat v |^2  \right) }{( 1 + \hat v \cdot \o )^2|y-x |^2}  f   ( t - |y-x|, y, v )  dv dy .
\end{split}
\Ee
where we have used that, from \cite{GS,Glassey},
\Be \notag
\begin{split}
&    \p_{y_j }  \left( \frac{ ( \o \times \hat v  ) \hat v_j }{ ( 1 + \hat v \cdot \o )|y-x | } \right) 
 \\ = &  \frac{   ( \o \times \hat v  ) \left( - ( \o \cdot \hat v )( 1 + \o \cdot \hat v ) -  ( \o \cdot \hat v ) ( 1 + \o \cdot \hat v )  -  |\hat v |^2  +  ( \o \cdot \hat v )^2 \right) }{( 1 + \hat v \cdot \o )^2|y-x |^2}
 \\ = &    \frac{   ( \o \times \hat v  ) \left( - 2 ( \o \cdot \hat v )   -  |\hat v |^2  -   ( \o \cdot \hat v )^2 \right) }{( 1 + \hat v \cdot \o )^2|y-x |^2},
 \end{split}
\Ee
and
\Be \notag
\begin{split}
& - \nabla_{y } (   \frac{1}{ |y-x|  }  ) \times \hat v +    \p_{y_j }  \left( \frac{ ( \o \times \hat v  ) \hat v_j }{ ( 1 + \hat v \cdot \o )|y-x | } \right)   
\\ & =  \frac{   ( \o \times \hat v  ) \left( ( 1 + \hat v \cdot \o )^2 - 2 ( \o \cdot \hat v )   -  |\hat v |^2  -   ( \o \cdot \hat v )^2 \right) }{( 1 + \hat v \cdot \o )^2|y-x |^2}
 =   \frac{   ( \o \times \hat v  ) \left( 1   -  |\hat v |^2  \right) }{( 1 + \hat v \cdot \o )^2|y-x |^2}.
\end{split}
\Ee
Now we consider \ref{Bexpanmajor2}. From \ref{ST_lower},  
\begin{align}
\notag \ref{Bexpanmajor2}
& =  \iota_i \int_{B(x;t)   \cap \{ y_3 < 0 \} } \int_{\mathbb R^3} \frac{(  \nabla_x f \times \hat v )_i ( t - |y-x|, \bar y, v ) }{|y-x| } dy
\\ \label{BlowerS} & =   \iota_i  \int_{B(x;t)  \cap \{ y_3 < 0 \} } \int_{\mathbb R^3}   \frac{ (\bar \o \times \hat v  )_i }{ 1+ \hat{v} \cdot \bar \o }  S  f ( t - |y-x|, \bar y ,v  ) dv \frac{ dy} {|y-x|}
\\ \label{BlowerT} &  + \iota_i \int_{B(x;t)  \cap \{ y_3 < 0 \} } \int_{\mathbb R^3}   \left(  ( \bar T \times \hat v )_i -  \frac{ ( \hat v \cdot  \bar T )  ( \bar \o \times \hat v )_i }{1+ \hat{v} \cdot  \bar \o}\right)    f ( t - |y-x|, \bar y, v ) dv \frac{ dy} {|y-x|}.
\end{align} 
As getting \ref{BupperS1}, we derive that, with $a_i^B$ of \ref{aB}, \ref{BlowerS} equals 
\Be\label{lowerS1}
\iota_i  \int_{B(x;t) \cap \{y_3 < 0\}}  \int_{\R^3}   a^B_i ( v ,\bar{\o}  ) \cdot (E + \hat v \times B - g\mathbf e_3 )  
 f (t-|y-x|,\bar y ,v) \dd v \frac{ \dd y}{|y-x|}.
\Ee

 For \ref{BlowerT}, applying \ref{def:T-} and the integration by parts, we derive that \ref{BlowerT} equals 
\Be \label{HSlowerhalf2}
\begin{split}
 &   \iota_i  \int_{ \p B(x; t) \cap \{ y_3 < 0 \} } \int_{\mathbb R^3}    (\bar \o \times \hat v)_i  \left( 1   - \frac{  \hat v \cdot \bar \o  }{1+ \hat{v} \cdot  \bar \o }\right)     f ( 0, \bar y, v ) dv \frac{ dS_y} {t}
\\  + &  \iota_i \int_{ B(x;t)  \cap \{ y_3 = 0 \} } \int_{\mathbb R^3}      \left( -(e_3 \times \hat v )_i  + \frac{  ( \bar \o \times \hat v )_i  }{1+ \hat{v} \cdot  \bar \o}  ( \hat v \cdot e_3 )  \right)   \frac{  f ( t- |y-x|, y_\parallel, 0, v ) }{|y-x|} dv d y_\parallel
\\ + & \iota_i \int_{B(x;t)  \cap \{ y_3 < 0 \} }  \int_{\mathbb R^3}  \left(    \frac{   ( \bar \o \times \hat v  )_i \left( 1   -  |\hat v |^2  \right) }{( 1 + \hat v \cdot \bar \o )^2|y-x |^2}  \right) f   ( t - |y-x|, \bar y, v )  dv dy, 
\end{split}
\Ee
where we have used the direct computation
\Be \notag
\begin{split}
  \iota_j  \p_{y_j }  \left( \frac{ ( \bar \o \times \hat v  ) \hat v_j }{ ( 1 + \hat v \cdot \bar \o )|y-x | } \right) = &  \frac{   ( \bar \o \times \hat v  ) \left( - ( \hat v \cdot \bar \o )( 1 + \hat v \cdot \bar \o ) -   \hat v \cdot \bar \o - | \hat v |^2 \right) }{( 1 + \hat v \cdot \bar \o )^2|y-x |^2}
 \\ = &    \frac{   ( \bar \o \times \hat v  ) \left( - 2 ( \hat v \cdot \bar \o )   -  |\hat v |^2  -   ( \hat v \cdot \bar \o )^2 \right) }{( 1 +\hat v \cdot \bar \o )^2|y-x |^2},
 \end{split}
\Ee
and
\Be  \notag
\begin{split}
&-   \overline{ \nabla_y (|y-x|^{-1} ) }  \times \hat v  +   \iota_j  \p_{y_j }  \left( \frac{ ( \bar \o \times \hat v  ) \hat v_j }{ ( 1 + \hat v \cdot \bar \o )|y-x | } \right)
\\ & =  \frac{   ( \bar \o \times \hat v  ) \left( ( 1 + \hat v \cdot \bar \o )^2 - 2 ( \hat v \cdot \bar \o )   -  |\hat v |^2  -   ( \hat v \cdot \bar \o )^2 \right) }{( 1 + \hat v \cdot \bar \o )^2|y-x |^2}
 =   \frac{   ( \bar \o \times \hat v  ) \left( 1   -  |\hat v |^2  \right) }{( 1 + \hat v \cdot \bar \o )^2|y-x |^2}.
\end{split}
\Ee

Collecting the terms, we conclude the following formula:
\begin{proposition}
\Be
\begin{split}
&  B_i(t,x )  
=    \frac{1}{4 \pi t^2} \int_{\p B(x;t) \cap \{ y_3 > 0 \} }  \left( t \p_t  B_{0i}( y ) +  B_{0i} (y) + \nabla  B_{0i} (y) \cdot (y-x) \right) dS_y
\\  
& + \frac{\iota_i}{4 \pi t^2} \int_{\p B(x;t) \cap \{ y_3 < 0 \} }  \big(  t \p_t  B_{0i}( \bar y ) +  B_{0i} (\bar y) + \nabla   B_{0i} (\bar y) \cdot (\bar y - \bar x)  \big) dS_y
\\ 
& + \int_{B(x;t)  \cap \{ y_3 > 0 \} }  \int_{\mathbb R^3}   \frac{   ( \o \times \hat v  )_i \left( 1   -  |\hat v |^2  \right) }{( 1 + \hat v \cdot \o )^2|y-x |^2} f   ( t - |y-x|, y, v )  dv dy
\\ 
& +  \int_{B(x;t)  \cap \{ y_3 < 0 \} }   \int_{\mathbb R^3}  \iota_i \frac{   ( \bar \o \times \hat v  )_i \left( 1   -  |\hat v |^2  \right) }{( 1 + \hat v \cdot \bar \o )^2|y-x |^2} f   ( t - |y-x|, \bar y, v )  dv dy
\\ 
& +    \int_{B(x;t) \cap \{y_3 > 0\}}  \int_{\R^3}   a^B_i ( v ,\o ) \cdot (E + \hat v \times B - g\mathbf e_3 )   f (t-|y-x|,y ,v) \dd v \frac{ \dd y}{|y-x|}
\\ 
& +   \int_{B(x;t) \cap \{y_3 < 0\}}  \int_{\R^3}  \iota_i a^B_i ( v ,\bar \o ) \cdot (E + \hat v \times B - g\mathbf e_3 )   f (t-|y-x|,\bar y ,v) \dd v \frac{ \dd y}{|y-x|}
\\  
& +   \int_{ B(x;t)  \cap \{ y_3 = 0 \} } \int_{\mathbb R^3}      \left( -(e_3 \times \hat v)_i   + \frac{ ( \o \times \hat v)_i   \hat v_3 }{1+ \hat{v} \cdot  \o}   \right)     f ( t- |y-x|, y_\parallel, 0, v ) dv \frac{ dy_\parallel} {|y-x|}
\\ 
& +  \int_{ B(x;t)  \cap \{ y_3 = 0 \} } \int_{\mathbb R^3}  \iota_i    \left( - (e_3 \times \hat v )_i  + \frac{ ( \bar \o \times \hat v)_i  \hat v_3   }{1+ \hat{v} \cdot  \bar \o}    \right)     f ( t- |y-x|, y_\parallel, 0, v ) dv \frac{ dy_\parallel} {|y-x|}
\\ 
& + \int_{\p B(x;t) \cap \{ y_3 > 0 \} } \int_{\mathbb R^3}   \left(  \frac{  (  \o \times \hat v)_i   }{1+ \hat{v} \cdot  \o}\right)     f ( 0, y, v ) dv \frac{ dS_y} {t}
\\ 
& +  \int_{\p B(x;t) \cap \{ y_3 < 0 \} } \int_{\mathbb R^3} \iota_i  \left(  \frac{    ( \bar \o \times \hat v )_i   }{1+ \hat{v} \cdot  \bar \o }\right)     f ( 0, \bar y, v ) dv \frac{ dS_y} {t}
\\ 
& + (-1)^i  2 (1 - \delta_{i3} )  \int_{ B(x;t) \cap \{y_3 =0\}}  \int_{\mathbb R^3 }  \frac{ \hat v_{\underline i }  f (t - |y-x |, y_\parallel, 0, v )}{|y-x | }  dv d  S_y.
\end{split}
\Ee
\end{proposition}

 \hide
\subsection{Derivatives of the Electronic field in a half space}
We calculate the derivatives of our representation of $E$. Let's first calculate $\p_t E$ by taking $\p_t$ derivative to each term in \ref{Eesttat0pos}-\ref{Eest3bdrycontri}.

Using change of variables $z = y- x$, and using polar coordinates we have
\[
\begin{split}
\ref{Eesttat0pos}  = &  \frac{1}{ 4 \pi t^2} \int_{ \{ |z|  = t \}  \cap \{ z_3 + x_3 > 0 \} }  \left( t \p_t  E_{0i}( x+ z  ) +  E_{0i} (x+z) + \nabla  E_{0i} (x+z) \cdot z \right) dS_z
\\ = & \frac{1}{4\pi}  \int_0^{\arccos( -x_3 /t ) } \int_0^{2 \pi } \left( t \p_t  E_{0i}( x+ z  ) +  E_{0i} (x+z) + \nabla  E_{0i} (x+z) \cdot z \right) \sin \phi d\theta d\phi,
\end{split}
\]
where $z = \begin{bmatrix} t sin \phi \cos \theta \\ t \sin \phi \sin \theta \\ t \cos \phi \end{bmatrix} $. We do the same for \ref{Eesttat0neg} by writing
\[
\begin{split}
\ref{Eesttat0neg}  = & - \frac{\iota_i}{4\pi}  \int_{\arccos( -x_3 /t ) }^0 \int_0^{2 \pi } \left( t \p_t  E_{0i}( \bar x + \bar z  ) +  E_{0i} (\bar x + \bar z) + \nabla  E_{0i} (\bar x + \bar z) \cdot  \bar z \right) \sin \phi d\theta d\phi.
\end{split}
\]
Thus 
\Be \label{Eesttat0posest1}
\begin{split}
& \frac{ \p }{\p_t} \ref{Eesttat0pos}  + \frac{ \p }{\p_t} \ref{Eesttat0neg} 
\\ = &  \frac{1}{4\pi}  \int_0^{\arccos( -x_3 /t ) } \int_0^{2 \pi } \big(  \p_t  E_{0i}( x+ z  )+ t  \nabla \p_t E_{0i}( x+ z  ) \cdot \frac{z}{t}  + \nabla E_{0,i} (x+z) \cdot \frac{z}{t}
\\ & \quad \quad \quad \quad \quad  \quad \quad \quad \quad \quad +\frac{z}{t}  \cdot \nabla^2  E_{0,i} (x+z) \cdot z + \nabla E_0(x+z) \cdot \frac{z}{t}  \big) \sin \phi d\theta d\phi
\\ & - \frac{1}{4\pi } \int_0^{2 \pi }  \left( t \p_t  E_{0,i}( x_\parallel+ z_\parallel,0 ) +  E_{0,i} (x_\parallel+ z_\parallel,0) + \nabla  E_{0,i} (x_\parallel+ z_\parallel,0) \cdot z \right)  \frac{x_3}{t^2 }  d\theta
\\ & -  \frac{\iota_i}{4\pi}  \int_{\arccos( -x_3 /t ) }^{\pi } \int_0^{2 \pi } \big(  \p_t  E_{0,i}( \bar x + \bar z  )+ t  \nabla \p_t E_{0,i}( \bar x + \bar z  ) \cdot \frac{\bar z}{t}  + \nabla E_{0,i} (\bar x + \bar z) \cdot \frac{\bar z}{t}
\\ & \quad \quad \quad \quad \quad  \quad \quad \quad \quad \quad +\frac{\bar z}{t}  \cdot \nabla^2  E_{0,i} (\bar x + \bar z) \cdot \bar z + \nabla E_0(\bar x + \bar z) \cdot \frac{\bar z}{t}  \big) \sin \phi d\theta d\phi
\\ & - \frac{\iota_i }{4\pi } \int_0^{2 \pi }  \left( t \p_t  E_{0,i}( x_\parallel+  z_\parallel,0 ) +  E_{0,i} (x_\parallel+  z_\parallel,0) + \nabla  E_{0,i} (x_\parallel+  z_\parallel,0) \cdot \bar z \right)  \frac{x_3}{t^2 }  d\theta
.
\end{split}
\Ee

Next, using change of variables $z = y- x$, we have
\Be \label{ptEestbulkpos1}
\begin{split}
\frac{\p}{\p t }  \ref{Eestbulkpos}  = &  \frac{\p}{\p t } \left(  \int_{ \{ |z| < t \}  \cap \{x_3 + z_3 > 0\}}  \int_{\R^3}\frac{ (|\hat{v}|^2-1  )(\hat{v}_i +  \o_i ) }{|z|^2 (1+  \hat v \cdot \o )^2} f(t-|z|,x+z ,v)\dd v \dd z \right) 
\\  = &  \int_{ \{ |z| < t \}  \cap \{x_3 + z_3 > 0\}}  \int_{\R^3}\frac{ (|\hat{v}|^2-1  )(\hat{v}_i +  \o_i ) }{|z|^2 (1+  \hat v \cdot \o )^2} \p_t f(t-|z|,x+z ,v)\dd v \dd z
\\ & + \int_{ \{ |z| = t \}  \cap \{x_3 + z_3 > 0\}}  \int_{\R^3}\frac{ (|\hat{v}|^2-1  )(\hat{v}_i +  \o_i ) }{t^2 (1+  \hat v \cdot \o )^2} f(0,x+z ,v)\dd v \dd S_z.
\end{split}
\Ee

From the equation \ref{VMfrakF}, we write $\p_t f = -  \hat v \cdot \nabla_x f -(E + \hat v \times B - g \mathbf{e}_3 ) \cdot \nabla_v f $. Then from integration by parts in $v$, and that $ \nabla_v \cdot ( \hat v \times B) =0$, we get
\Be \label{ptEestbulkpos2}
\begin{split}
&  \int_{ \{ |z| < t \}  \cap \{x_3 + z_3 > 0\}}  \int_{\R^3}\frac{ (|\hat{v}|^2-1  )(\hat{v}_i +  \o_i ) }{|z|^2 (1+  \hat v \cdot \o )^2} \p_t f(t-|z|,x+z ,v)\dd v \dd z
 \\ = & \int_{ \{ |z| < t \}  \cap \{x_3 + z_3 > 0\}}  \int_{\R^3}\frac{ (|\hat{v}|^2-1  )(\hat{v}_i +  \o_i ) }{|z|^2 (1+  \hat v \cdot \o )^2} \left(  -  \hat v \cdot \nabla_x f -(E + \hat v \times B - g \mathbf{e}_3 ) \cdot \nabla_v f \right) (t-|z|,x+z ,v)\dd v \dd z
 \\ = & \int_{ \{ |z| < t \}  \cap \{x_3 + z_3 > 0\}}   \int_{\R^3} - \frac{ (|\hat{v}|^2-1  )(\hat{v}_i +  \o_i ) }{|z|^2 (1+  \hat v \cdot \o )^2}    \hat v \cdot \nabla_x f(t-|z|,x+z ,v)\dd v \dd z
 \\ & + \int_{ \{ |z| < t \}  \cap \{x_3 + z_3 > 0\}}  \int_{\R^3} \nabla_v  \left( \frac{ (|\hat{v}|^2-1  )(\hat{v}_i +  \o_i ) }{ (1+  \hat v \cdot \o )^2} \right)  \cdot (E + \hat v \times B - g \mathbf{e}_3 ) \frac{ f (t-|z|,x+z ,v)}{|z|^2 } \dd v \dd z.
\end{split}
\Ee
Therefore,
\Be \label{ptEestbulkpos11}
\begin{split}
\frac{\p}{\p t }  \ref{Eestbulkpos}  = &  \int_{ \{ |z| < t \}  \cap \{x_3 + z_3 > 0\}}   \int_{\R^3} - \frac{ (|\hat{v}|^2-1  )(\hat{v}_i +  \o_i ) }{|z|^2 (1+  \hat v \cdot \o )^2}    \hat v \cdot \nabla_x f(t-|z|,x+z ,v)\dd v \dd z
 \\ & + \int_{ \{ |z| < t \}  \cap \{x_3 + z_3 > 0\}}  \int_{\R^3} \nabla_v  \left( \frac{ (|\hat{v}|^2-1  )(\hat{v}_i +  \o_i ) }{ (1+  \hat v \cdot \o )^2} \right)  \cdot (E + \hat v \times B - g \mathbf{e}_3 ) \frac{ f (t-|z|,x+z ,v)}{|z|^2 } \dd v \dd z
\\ & + \int_{ \{ |z| = t \}  \cap \{x_3 + z_3 > 0\}}  \int_{\R^3}\frac{ (|\hat{v}|^2-1  )(\hat{v}_i +  \o_i ) }{t^2 (1+  \hat v \cdot \o )^2} f(0,x+z ,v)\dd v \dd S_z.
\end{split}
\Ee
And by direct computation,
\Be
\nabla_v  \left( \frac{ (|\hat{v}|^2-1  )(\hat{v}_i +  \o_i ) }{ (1+  \hat v \cdot \o )^2} \right) = \frac{1}{(1 + \hat v \cdot \o )^2} \left(\frac{2 \hat v (\hat v_i + \o_i ) - \mathbf e_i + \hat v_i \hat v }{ \langle v \rangle^3 } \right) +  \frac{2}{ (1 + \hat v \cdot \o )^3 } \left(  \frac{ (\o - (\o \cdot \hat v ) \hat v ) (\hat v_i +  \o_i ) }{ \langle v \rangle^3 } \right).
\Ee
Doing the same for \ref{Eestbulkneg}, we get
\Be \label{ptEestbulkneg11}
\begin{split}
\frac{\p}{\p t }  \ref{Eestbulkneg}  = & \iota_i \int_{ \{ |z| < t \}  \cap \{x_3 + z_3 < 0\}}   \int_{\R^3}  \frac{ (|\hat{v}|^2-1  )(\hat{v}_i + \bar \o_i ) }{|z|^2 (1+  \hat v \cdot \bar \o )^2}    \hat v \cdot \nabla_x f(t-|z|,x+z ,v)\dd v \dd z
 \\ & - \iota_i \int_{ \{ |z| < t \}  \cap \{x_3 + z_3 < 0\}}  \int_{\R^3} \nabla_v  \left( \frac{ (|\hat{v}|^2-1  )(\hat{v}_i + \bar \o_i ) }{ (1+  \hat v \cdot \bar \o )^2} \right)  \cdot (E + \hat v \times B - g \mathbf{e}_3 ) \frac{ f (t-|z|,x+z ,v)}{|z|^2 } \dd v \dd z
\\ & - \iota_i \int_{ \{ |z| = t \}  \cap \{x_3 + z_3 < 0\}}  \int_{\R^3}\frac{ (|\hat{v}|^2-1  )(\hat{v}_i + \bar \o_i ) }{t^2 (1+  \hat v \cdot \bar \o )^2} f(0,x+z ,v)\dd v \dd S_z.
\end{split}
\Ee

Next, we have from change of variables $ z = y -x$,
\Be \label{ptEestSpos1}
\begin{split}
\frac{ \p}{\p t } \ref{EestSpos} = &  \frac{ \p}{ \p t } \left( \int_{ \{ |z| < t \} \cap \{ z_3+x_3 > 0 \} } \int_{\mathbb R^3 }  a^E_i (v ,\o ) \cdot (E + \hat v \times B - g\mathbf e_3 )   f (t-|z|,x+z ,v) \dd v \frac{ \dd z}{|z|}  \right)
\\ = &  \int_{ \{ |z| < t \} \cap \{ z_3+x_3 > 0 \} } \int_{\mathbb R^3 } a^E_i (v ,\o ) \cdot (E + \hat v \times B - g\mathbf e_3 )   \p_t f (t-|z|,x+z ,v) \dd v \frac{ \dd z}{|z|}  
\\ & +  \int_{ \{ |z| < t \} \cap \{ z_3+x_3 > 0 \} } \int_{\mathbb R^3 } a^E_i (v ,\o ) \cdot (\p_t E + \hat v \times \p_t B - g\mathbf e_3 )  f (t-|z|,x+z ,v) \dd v \frac{ \dd z}{|z|} 
\\ & + \int_{ \{ |z| = t \} \cap \{ z_3+x_3 > 0 \} } \int_{\mathbb R^3 } a^E_i (v ,\o ) \cdot (E + \hat v \times B - g\mathbf e_3 )   f (0,x+z ,v) \dd v \frac{ \dd S_z}{t} ,
\end{split}
\Ee
We again write $\p_t f = -  \hat v \cdot \nabla_x f -(E + \hat v \times B - g \mathbf{e}_3 ) \cdot \nabla_v f $. Then from integration by parts in $v$,
\Be \label{ptEestSpos2}
\begin{split}
& \int_{ \{ |z| < t \} \cap \{ z_3+x_3 > 0 \} } \int_{\mathbb R^3 } a^E_i (v ,\o ) \cdot (E + \hat v \times B - g\mathbf e_3 )   \p_t f (t-|z|,x+z ,v) \dd v \frac{ \dd z}{|z|}
\\ = & \int_{ \{ |z| < t \} \cap \{ z_3+x_3 > 0 \} } \int_{\mathbb R^3 } -  a^E_i (v ,\o ) \cdot (E + \hat v \times B - g\mathbf e_3 ) \hat v \cdot \nabla_x f  (t-|z|,x+z ,v) \dd v \frac{ \dd z}{|z|}
\\ & + \int_{ \{ |z| < t \} \cap \{ z_3+x_3 > 0 \} } \int_{\mathbb R^3 } \nabla_v \left(  a^E_i (v ,\o ) \cdot (E + \hat v \times B - g\mathbf e_3 ) \right) \cdot (E + \hat v \times B - g \mathbf{e}_3 )  f  (t-|z|,x+z ,v) \dd v \frac{ \dd z}{|z|}.
\end{split}
\Ee
So
\Be \label{ptEestSpos11}
\begin{split}
& \frac{ \p}{\p t } \ref{EestSpos} 
\\ = &  \int_{ \{ |z| < t \} \cap \{ z_3+x_3 > 0 \} } \int_{\mathbb R^3 } -  a^E_i (v ,\o ) \cdot (E + \hat v \times B - g\mathbf e_3 ) \hat v \cdot \nabla_x f  (t-|z|,x+z ,v) \dd v \frac{ \dd z}{|z|}
\\ & + \int_{ \{ |z| < t \} \cap \{ z_3+x_3 > 0 \} } \int_{\mathbb R^3 } \nabla_v \left(  a^E_i (v ,\o ) \cdot (E + \hat v \times B - g\mathbf e_3 ) \right) \cdot (E + \hat v \times B - g \mathbf{e}_3 )  f  (t-|z|,x+z ,v) \dd v \frac{ \dd z}{|z|}  
\\ & +  \int_{ \{ |z| < t \} \cap \{ z_3+x_3 > 0 \} } \int_{\mathbb R^3 } a^E_i (v ,\o ) \cdot (\p_t E + \hat v \times \p_t B - g\mathbf e_3 )  f (t-|z|,x+z ,v) \dd v \frac{ \dd z}{|z|} 
\\ & + \int_{ \{ |z| = t \} \cap \{ z_3+x_3 > 0 \} } \int_{\mathbb R^3 } a^E_i (v ,\o ) \cdot (E + \hat v \times B - g\mathbf e_3 )   f (0,x+z ,v) \dd v \frac{ \dd S_z}{t}.
\end{split}
\Ee
Similarly,
\Be \label{ptEestSneg11}
\begin{split}
& \frac{ \p}{\p t } \ref{EestSneg} 
\\= & \iota_i \int_{ \{ |z| < t \} \cap \{ z_3+x_3 < 0 \} } \int_{\mathbb R^3 }   a^E_i (v ,\bar \o ) \cdot (E + \hat v \times B - g\mathbf e_3 ) \hat v \cdot \nabla_x f  (t-|z|,\bar x + \bar z ,v) \dd v \frac{ \dd z}{|z|}
\\ & - \iota_i \int_{ \{ |z| < t \} \cap \{ z_3+x_3 < 0 \} } \int_{\mathbb R^3 } \nabla_v \left(  a^E_i (v ,\bar \o ) \cdot (E + \hat v \times B - g\mathbf e_3 ) \right) \cdot (E + \hat v \times B - g \mathbf{e}_3 )  f  (t-|z|,\bar x + \bar z  ,v) \dd v \frac{ \dd z}{|z|}  
\\ & - \iota_i  \int_{ \{ |z| < t \} \cap \{ z_3+x_3 < 0 \} } \int_{\mathbb R^3 } a^E_i (v ,\bar \o ) \cdot (\p_t E + \hat v \times \p_t B - g\mathbf e_3 )  f (t-|z|,\bar x + \bar z ,v) \dd v \frac{ \dd z}{|z|} 
\\ & - \iota_i \int_{ \{ |z| = t \} \cap \{ z_3+x_3 < 0 \} } \int_{\mathbb R^3 } a^E_i (v ,\bar \o ) \cdot (E + \hat v \times B - g\mathbf e_3 )   f (0,\bar x + \bar z ,v) \dd v \frac{ \dd S_z}{t}.
\end{split}
\Ee
Next, using change of variables $z_\parallel = y_\parallel - x_\parallel$, we have
\Be \label{ptEestbdrypos1}
\begin{split}
& \frac{ \p }{\p t } \ref{Eestbdrypos}
\\ = &  \frac{ \p }{\p t } \left(   \int_{ |z_\parallel | < \sqrt{ t^2 - x_3^2 } }  \int_{\R^3} \left( \delta_{i3 } -  \frac{(\o_i + \hat{v}_i)\hat{v}_3}{1+ \hat v \cdot \o } \right) f (t- ( |z_\parallel |^2 + x_3^2 )^{1/2} ,z_\parallel + x_\parallel, 0 ,v) \dd v\frac{ \dd z_\parallel}{(|z_\parallel |^2 + x_3^2)^{1/2}} \right)
\\ = &  \int_{ |z_\parallel | < \sqrt{ t^2 - x_3^2 } }  \int_{\R^3} \left( \delta_{i3 } -  \frac{(\o_i + \hat{v}_i)\hat{v}_3}{1+ \hat v \cdot \o } \right)\p_t f (t- ( |z_\parallel |^2 + x_3^2 )^{1/2} ,z_\parallel + x_\parallel, 0 ,v) \dd v\frac{ \dd z_\parallel}{(|z_\parallel |^2 + x_3^2)^{1/2}}
\\ & + \int_0^{2\pi }   \int_{\R^3} \left( \delta_{i3 } -  \frac{(\o_i + \hat{v}_i)\hat{v}_3}{1+ \hat v \cdot \o } \right) f ( 0  ,z_\parallel + x_\parallel, 0 ,v) \frac{t}{\sqrt{t^2 -x_3^2 } } \sqrt{t^2 -x_3^2 }  \dd v\frac{ \dd \theta}{t}.
\end{split}
\Ee
We again write $\p_t f = -  \hat v \cdot \nabla_x f -(E + \hat v \times B - g \mathbf{e}_3 ) \cdot \nabla_v f $ and use integration by parts in $v$ to get
\Be \label{ptEestbdrypos2}
\begin{split}
&  \int_{ |z_\parallel | < \sqrt{ t^2 - x_3^2 } }  \int_{\R^3} \left( \delta_{i3 } -  \frac{(\o_i + \hat{v}_i)\hat{v}_3}{1+ \hat v \cdot \o } \right)\p_t f (t- ( |z_\parallel |^2 + x_3^2 )^{1/2} ,z_\parallel + x_\parallel, 0 ,v) \dd v\frac{ \dd z_\parallel}{(|z_\parallel |^2 + x_3^2)^{1/2}}
 \\ = &  \int_{ |z_\parallel | < \sqrt{ t^2 - x_3^2 } }  \int_{\R^3} \left( \delta_{i3 } -  \frac{(\o_i + \hat{v}_i)\hat{v}_3}{1+ \hat v \cdot \o } \right) \frac{\hat v \cdot \nabla_x f (t- ( |z_\parallel |^2 + x_3^2 )^{1/2} ,z_\parallel + x_\parallel, 0 ,v) }{ (|z_\parallel |^2 + x_3^2)^{1/2}} \dd v \dd z_\parallel
 \\ & + \int_{ |z_\parallel | < \sqrt{ t^2 - x_3^2 } }  \int_{\R^3} \nabla_v \left( \delta_{i3 } -  \frac{(\o_i + \hat{v}_i)\hat{v}_3}{1+ \hat v \cdot \o } \right) \cdot  ( E + \hat v \times B -  g \mathbf e_3 ) \frac{ f (t- ( |z_\parallel |^2 + x_3^2 )^{1/2} ,z_\parallel + x_\parallel, 0 ,v) }{ (|z_\parallel |^2 + x_3^2)^{1/2}} \dd v \dd z_\parallel.
\end{split}
\Ee
Thus,
\Be \label{ptEestbdrypos11}
\begin{split}
& \frac{ \p }{\p t } \ref{Eestbdrypos}
\\ = &  \int_{ |z_\parallel | < \sqrt{ t^2 - x_3^2 } }  \int_{\R^3} \left( \delta_{i3 } -  \frac{(\o_i + \hat{v}_i)\hat{v}_3}{1+ \hat v \cdot \o } \right) \frac{\hat v \cdot \nabla_x f (t- ( |z_\parallel |^2 + x_3^2 )^{1/2} ,z_\parallel + x_\parallel, 0 ,v) }{ (|z_\parallel |^2 + x_3^2)^{1/2}} \dd v \dd z_\parallel
 \\ & + \int_{ |z_\parallel | < \sqrt{ t^2 - x_3^2 } }  \int_{\R^3} \nabla_v \left( \delta_{i3 } -  \frac{(\o_i + \hat{v}_i)\hat{v}_3}{1+ \hat v \cdot \o } \right) \cdot  ( E + \hat v \times B -  g \mathbf e_3 ) \frac{ f (t- ( |z_\parallel |^2 + x_3^2 )^{1/2} ,z_\parallel + x_\parallel, 0 ,v) }{ (|z_\parallel |^2 + x_3^2)^{1/2}} \dd v \dd z_\parallel
\\ & + \int_0^{2\pi }   \int_{\R^3} \left( \delta_{i3 } -  \frac{(\o_i + \hat{v}_i)\hat{v}_3}{1+ \hat v \cdot \o } \right) f ( 0  ,z_\parallel + x_\parallel, 0 ,v)  \dd v \dd \theta.
\end{split}
\Ee
Similarly, 
\Be \label{ptEestbdryneg11}
\begin{split}
& \frac{ \p }{\p t } \ref{Eestbdryneg}
\\ = & - \iota_i \int_{ |z_\parallel | < \sqrt{ t^2 - x_3^2 } }  \int_{\R^3} \left( \delta_{i3 } -  \frac{(\bar \o_i + \hat{v}_i)\hat{v}_3}{1+ \hat v \cdot \o } \right) \frac{\hat v \cdot \nabla_x f (t- ( |z_\parallel |^2 + x_3^2 )^{1/2} ,z_\parallel + x_\parallel, 0 ,v) }{ (|z_\parallel |^2 + x_3^2)^{1/2}} \dd v \dd z_\parallel
 \\ &  - \iota_i  \int_{ |z_\parallel | < \sqrt{ t^2 - x_3^2 } }  \int_{\R^3} \nabla_v \left( \delta_{i3 } -  \frac{(\bar \o_i + \hat{v}_i)\hat{v}_3}{1+ \hat v \cdot \o } \right) \cdot  ( E + \hat v \times B -  g \mathbf e_3 ) \frac{ f (t- ( |z_\parallel |^2 + x_3^2 )^{1/2} ,z_\parallel + x_\parallel, 0 ,v) }{ (|z_\parallel |^2 + x_3^2)^{1/2}} \dd v \dd z_\parallel
\\ &  - \iota_i  \int_0^{2\pi }   \int_{\R^3} \left( \delta_{i3 } -  \frac{(\bar \o_i + \hat{v}_i)\hat{v}_3}{1+ \hat v \cdot \o } \right) f ( 0  ,z_\parallel + x_\parallel, 0 ,v)  \dd v \dd \theta.
\end{split}
\Ee

Next, using the change of variables $z = y-x$ and then the polar coordinate for $z$, we have
\Be \label{ptEestinitialpos1}
\begin{split}
 \ref{Eestinitialpos} =  & -  \int_{ \{ |z| = t\} \cap \{  z_3 + x_3 >0 \} } \int_{\R^3}  \sum_j \o_j  \left(\delta_{ij} - \frac{(\o_i + \hat{v}_i)\hat{v}_j}{1+ \hat v \cdot \o }\right) f(0,x+z ,v) \dd v \frac{\dd S_z}{|z|}
 \\ = & -   \int_0^{\arccos(-x_3/t )} \int_0^{2\pi}  \int_{\R^3}   \sum_j \o_j  \left(\delta_{ij} - \frac{(\o_i + \hat{v}_i)\hat{v}_j}{1+ \hat v \cdot \o }\right)  \frac{ f(0,x+z ,v)}{t}  t^2 \sin \phi \dd v d \theta d \phi.
  \end{split}
\Ee
Thus
\Be \label{ptEestinitialpos11}
\begin{split}
\frac{\p}{\p t }  \ref{Eestinitialpos} = & -  \int_0^{\arccos(-x_3/t )} \int_0^{2\pi}  \int_{\R^3}   \sum_j \o_j  \left(\delta_{ij} - \frac{(\o_i + \hat{v}_i)\hat{v}_j}{1+ \hat v \cdot \o }\right)  f(0,x+z ,v) \sin \phi  \dd v  d \theta d \phi
\\ & -   \int_0^{\arccos(-x_3/t )} \int_0^{2\pi}  \int_{\R^3}   \sum_j \o_j  \left(\delta_{ij} - \frac{(\o_i + \hat{v}_i)\hat{v}_j}{1+ \hat v \cdot \o }\right)( \nabla_x f(0,x+z ,v) \cdot z ) \sin \phi  \dd v  d \theta d \phi
\\ & -  \int_0^{2\pi}  \int_{\R^3}   \sum_j \o_j  \left(\delta_{ij} - \frac{(\o_i + \hat{v}_i)\hat{v}_j}{1+ \hat v \cdot \o }\right)  f(0,x_\parallel +z_\parallel, 0 ,v) \frac{-\sqrt{ 1- (x_3/t)^2 } }{\sqrt{ 1- (x_3/t)^2 } } \frac{-x_3}{t^2 } t \  \dd v  d \theta.
\end{split}
\Ee
Similarly, 
\Be \label{ptEestinitialneg11}
\begin{split}
\frac{\p}{\p t }  \ref{Eestinitialneg} = & \iota_i  \int_{\arccos(-x_3/t )}^0 \int_0^{2\pi}  \int_{\R^3}    \left(\delta_{ij} - \frac{( \bar \o_i + \hat{v}_i)\hat{v}_j}{1+ \hat v \cdot \bar \o }\right)  f(0,\bar x+ \bar z ,v) \sin \phi  \dd v  d \theta d \phi
\\ & +  \iota_i  \int_{\arccos(-x_3/t )}^0 \int_0^{2\pi}  \int_{\R^3}   \bar \o_j  \left(\delta_{ij} - \frac{(\bar \o_i + \hat{v}_i)\hat{v}_j}{1+ \hat v \cdot \bar \o }\right)( \nabla_x f(0,\bar x+ \bar z ,v) \cdot \bar z ) \sin \phi  \dd v  d \theta d \phi
\\ & -  \iota_i  \int_0^{2\pi}  \int_{\R^3}   \bar \o_j \left(\delta_{ij} - \frac{(\bar \o_i + \hat{v}_i)\hat{v}_j}{1+ \hat v \cdot \bar \o }\right)  f(0,x_\parallel +z_\parallel, 0 ,v)  \frac{x_3}{t }  \  \dd v  d \theta.
\end{split}
\Ee

Finally, for \ref{Eest3bdrycontri}, using change of variables $z_\parallel = y_\parallel - x_\parallel$, we have
\Be \label{ptEest6}
\begin{split}
 \frac{ \p}{\p t } \ref{Eest3bdrycontri}   = & - 2 \delta_{i3}   \int_{ |z_\parallel | < \sqrt{ t^2 - x_3^2 } }  \int_{\R^3}  \frac{\hat v \cdot \nabla_x f (t- ( |z_\parallel |^2 + x_3^2 )^{1/2} ,z_\parallel + x_\parallel, 0 ,v) }{ (|z_\parallel |^2 + x_3^2)^{1/2}} \dd v \dd z_\parallel
\\ & - 2 \delta_{i3} \int_0^{2\pi }   \int_{\R^3}  f ( 0  ,z_\parallel + x_\parallel, 0 ,v)  \dd v \dd \theta.
\end{split}
\Ee

Next, we calculate the spatial derivatives to $E(t,x)$ by taking $\frac{\p}{\p_{x_k}} $ to \ref{Eesttat0pos}-\ref{Eest3bdrycontri}. We have
\Be \label{pxEesttat0pos}
\begin{split}
& \frac{ \p}{\p_{x_k } }  \ref{Eesttat0pos} + \frac{ \p}{\p_{x_k } }  \ref{Eesttat0neg}  
\\ = &   \frac{1}{4\pi } \int_0^{\arccos(-x_3/t ) }  \int_{0 }^{2\pi}    \left( t \p_{x_k } \p_t  E_{0,i}( x + z  ) +  \p_{x_k }  E_{0,i} (x+z) +  \nabla  \p_{x_k }  E_{0,i} ( x+ z) \cdot z \right)  \sin \phi d\theta d \phi
\\ & + \frac{1}{4 \pi t}  \delta_{k3 } \int_0^{2\pi } (  t \p_t  E_{0,i}( x_\parallel + z_\parallel , 0  ) +  E_{0,i} (x_\parallel + z_\parallel , 0) + \nabla  E_{0,i} ( x_\parallel + z_\parallel , 0) \cdot z )  d\theta
\\ & -  \frac{\iota_i \iota_k}{4\pi}  \int_{\arccos( -x_3 /t ) }^{\pi } \int_0^{2 \pi } \big(   t \p_{x_k } \p_t  E_{0,i}( \bar x + \bar z  ) +  \p_{x_k }  E_{0,i} (\bar x + \bar z) +  \nabla  \p_{x_k }  E_{0,i} ( \bar x + \bar z) \cdot \bar z  \big) \sin \phi d\theta d\phi
\\ & + \frac{\iota_i }{4\pi t  } \delta_{k3 }  \int_0^{2 \pi }   \left( t \p_t  E_{0,i}( x_\parallel+  z_\parallel,0 ) +  E_{0,i} (x_\parallel+  z_\parallel,0) + \nabla  E_{0,i} (x_\parallel+  z_\parallel,0) \cdot \bar z \right)   d\theta
.
\end{split}
\Ee
Next, using the change of variables $ z = y-x$ we have
\Be \label{Eestbulkpos12}
\begin{split}
\frac{ \p}{\p_{x_k } }  \ref{Eestbulkpos} = & \int_{ \{ |z| < t  \} \cap \{z_3+x_3 > 0 \} } \int_{\R^3} \frac{ (|\hat{v}|^2-1 )(\hat{v}_i +  \o_i ) }{|z|^2 (1+  \hat v \cdot \o )^2} \p_{x_k}f (t-|z| ,x+z ,v) \dd v \dd z  
 \\& + \delta_{k3} \int_{ \{  (| z_\parallel  |^2 + |x_3|^2 )^{1/2} < t  \}  }   \int_{\R^3}\frac{ (|\hat{v}|^2-1  )(\hat{v}_i +  \o_i ) }{ ( |z_\parallel  |^2 + x_3^2 ) (1+  \hat v \cdot \o )^2} f(t-( |z_\parallel  |^2 + x_3^2 )^{1/2}, x_\parallel + z_\parallel , 0  ,v) \dd v \dd z_\parallel,
\end{split}
\Ee
\Be \label{Eestbulkneg12}
\begin{split}
\frac{ \p}{\p_{x_k } }  \ref{Eestbulkneg} = & -  \iota_i \iota_k \int_{ \{ |z | < t  \} \cap \{z_3 +x_3 < 0 \} } \int_{\R^3}  \frac{ (|\hat{v}|^2-1 )(\hat{v}_i + \bar \o_i ) }{|z|^2 (1+  \hat v \cdot \bar \o )^2} \p_{x_k}f (t-|z| , \bar x+ \bar z ,v) \dd v \dd z  
 \\& + \iota_i \delta_{k3} \int_{ \{  (| z_\parallel |^2 + |x_3|^2 )^{1/2} < t  \}  }   \int_{\R^3}\frac{ (|\hat{v}|^2-1  )(\hat{v}_i + \bar \o_i ) }{ ( | z_\parallel |^2 + x_3^2 ) (1+  \hat v \cdot \bar \o )^2} f(t-( |z_\parallel  |^2 + x_3^2 )^{1/2}, x_\parallel + z_\parallel , 0  ,v) \dd v \dd z_\parallel.
\end{split}
\Ee

Next, using the change of variables $z = y -x $ and taking $\frac{\p}{\p_{x_k } } $ derivative to \ref{EestSpos} and  \ref{EestSneg} we have
\Be \label{EestSpos12}
\begin{split}
\frac{ \p }{\p_{x_k } } \ref{EestSpos} = &  \int_{ \{ |z| < t  \} \cap \{z_3+x_3 > 0 \} }  \int_{\R^3}   a^E_i (v ,\o ) \cdot ( \p_{x_k } E +   \hat v  \times  \p_{x_k } B)   f (t-|z|,x+z ,v) \dd v \frac{ \dd z}{|z|}
\\  & +  \int_{ \{ |z| < t  \} \cap \{z_3+x_3 > 0 \} }  \int_{\R^3}   a^E_i (v ,\o ) \cdot (  E +   \hat v  \times B) \p_{x_k}  f (t-|z|,x+z ,v) \dd v \frac{ \dd z}{|z|}
\\ & + \delta_{k3} \int_{ \{  (| z_\parallel  |^2 + |x_3|^2 )^{1/2} < t  \}  }  \int_{\R^3}   a^E_i (v ,\o ) \cdot (  E +   \hat v  \times B)   \frac{  f (t-(| z_\parallel  |^2 + |x_3|^2 )^{1/2},x_\parallel +z_\parallel, 0 ,v) }{(| z_\parallel  |^2 + |x_3|^2 )^{1/2}} \dd v \dd z_\parallel,
\end{split}
\Ee
\Be \label{EestSneg2}
\begin{split}
 \frac{ \p }{\p_{x_k } } \ref{EestSneg} 
 = & - \iota_i \int_{ \{ |z| < t  \} \cap \{z_3+x_3 < 0 \} }  \int_{\R^3}   a^E_i (v ,\bar \o ) \cdot ( \p_{x_k } E +   \hat v  \times  \p_{x_k } B)   f (t-|z|,\bar x + \bar z ,v) \dd v \frac{ \dd z}{|z|}
\\  & - \iota_i \iota_k \int_{ \{ |z| < t  \} \cap \{z_3+x_3 < 0 \} }  \int_{\R^3}   a^E_i (v ,\bar \o ) \cdot (  E +   \hat v  \times B) \p_{x_k}  f (t-|z|,\bar x + \bar z ,v) \dd v \frac{ \dd z}{|z|}
\\ & + \iota_i \delta_{k3} \int_{ \{  (| z_\parallel  |^2 + |x_3|^2 )^{1/2} < t  \}  }   \int_{\R^3}   a^E_i (v ,\bar \o ) \cdot (  E +   \hat v  \times B)   \frac{f (t-(| z_\parallel  |^2 + |x_3|^2 )^{1/2},x_\parallel+z_\parallel , 0 ,v) }{(| z_\parallel  |^2 + |x_3|^2 )^{1/2}} \dd v  \dd z_\parallel.
\end{split}
\Ee

Next, using the change of variables $z_\parallel = y_\parallel - x_\parallel $ we have
\Be \label{Eestbdrypos1}
\begin{split}
& \frac{ \p }{\p_{x_k } } \ref{Eestbdrypos}
\\ = &  \frac{ \p }{\p_{x_k } } \left(   \int_{ \sqrt{t^2 - |z_\parallel |^2 } > x_3}  
\int_{\R^3} \left( \delta_{i3 } -  \frac{(\o_i + \hat{v}_i)\hat{v}_3}{1+ \hat v \cdot \o } \right) f (t- ( |z_\parallel |^2 + x_3^2 )^{1/2} ,z_\parallel + x_\parallel, 0 ,v) \dd v\frac{ \dd z_\parallel}{(|z_\parallel |^2 + x_3^2)^{1/2}} \right)
\\ = &  ( 1 - \delta_{k3} )   \int_{ \sqrt{t^2 - |z_\parallel |^2 } > x_3}  \int_{\R^3}   \left( \delta_{i3 } -\frac{(\o_i + \hat{v}_i)\hat{v}_3}{1+ \hat v \cdot \o } \right)  \p_{x_k } f (t- ( |z_\parallel |^2 + x_3^2 )^{1/2} ,z_\parallel + x_\parallel, 0 ,v) \dd v\frac{ \dd z_\parallel}{(|z_\parallel |^2 + x_3^2)^{1/2}}
\\ & + \delta_{k3 }  \int_{ \sqrt{t^2 - |z_\parallel |^2 } > x_3}  \int_{\R^3}   \left( \delta_{i3 } -\frac{(\o_i + \hat{v}_i)\hat{v}_3}{1+ \hat v \cdot \o } \right) f (t- ( |z_\parallel |^2 + x_3^2 )^{1/2} ,z_\parallel + x_\parallel, 0 ,v) \dd v \  \left( \frac{ -x_3}{(|z_\parallel |^2 + x_3^2)^{3/2}} \right)  \dd z_\parallel
\\ & + \delta_{k3}   \int_{ \sqrt{t^2 - |z_\parallel |^2 } > x_3}  \int_{\R^3} \left( \delta_{i3} - \frac{(\o_i + \hat{v}_i)\hat{v}_3}{1+ \hat v \cdot \o } \right)  \p_{t } f (t- ( |z_\parallel |^2 + x_3^2 )^{1/2} ,z_\parallel + x_\parallel, 0 ,v) \dd v\frac{ - x_3 }{(|z_\parallel |^2 + x_3^2)} \dd z_\parallel
\\ & -   \delta_{k3}   \int_{ \sqrt{t^2 - |z_\parallel |^2 } = x_3}   \int_{\R^3} \left( \delta_{i3} -  \frac{(\o_i + \hat{v}_i)\hat{v}_3}{1+ \hat v \cdot \o }  \right) f (t- ( |z_\parallel |^2 + x_3^2 )^{1/2} ,z_\parallel + x_\parallel, 0 ,v) \dd v  \frac{x_3 }{\sqrt{t^2 - x_3^2} }  \frac{ \dd S_{ z_\parallel } }{t}.
\end{split} 
\Ee 
Now from \ref{VMfrakF}, we write $\p_t f = -  \hat v \cdot \nabla_x f -(E + \hat v \times B - g \mathbf{e}_3 ) \cdot \nabla_v f $. Then integration by parts in $v$, and that $ \nabla_v \cdot ( \hat v \times B) =0$, we get
\Be \label{Eestbdrypos3}
\begin{split}
&    \int_{ \sqrt{t^2 - |z_\parallel |^2 } > x_3}  \int_{\R^3}  \left( \delta_{i3} - \frac{(\o_i + \hat{v}_i)\hat{v}_3}{1+ \hat v \cdot \o } \right)  \p_{t } f (t- ( |z_\parallel | + x_3^2 )^{1/2} ,z_\parallel + x_\parallel, 0 ,v) \dd v \frac{ x_3 }{(|z_\parallel |^2 + x_3^2)} \dd z_\parallel 
\\ =  &  -  \int_{ \sqrt{t^2 - |z_\parallel |^2 } > x_3}  \int_{\R^3} \left( \delta_{i3} - \frac{(\o_i + \hat{v}_i)\hat{v}_3}{1+ \hat v \cdot \o } \right)\left(  \hat v \cdot \nabla_x f (t- ( |z_\parallel | + x_3^2 )^{1/2} ,z_\parallel + x_\parallel, 0 ,v) \right) \dd v \frac{ x_3 }{(|z_\parallel |^2 + x_3^2)} \dd z_\parallel 
\\ & +    \int_{ \sqrt{t^2 - |z_\parallel |^2 } > x_3}  \int_{\R^3}  \nabla_v   \left( \delta_{i3} - \frac{(\o_i + \hat{v}_i)\hat{v}_3}{1+ \hat v \cdot \o } \right)  (E + \hat v \times B - g \mathbf{e}_3 )  \frac{x_3  f (t- ( |z_\parallel | + x_3^2 )^{1/2} ,z_\parallel + x_\parallel, 0 ,v) }{(|z_\parallel |^2 + x_3^2)} \dd v \dd z_\parallel. 
\end{split}
\Ee
Therefore,
\Be \label{Eestbdrypos21}
\begin{split}
& \frac{ \p }{\p_{x_k } } \ref{Eestbdrypos}
\\ = &  ( 1 - \delta_{k3} )   \int_{ \sqrt{t^2 - |z_\parallel |^2 } > x_3}  \int_{\R^3}   \left( \delta_{i3 } -\frac{(\o_i + \hat{v}_i)\hat{v}_3}{1+ \hat v \cdot \o } \right)  \p_{x_k } f (t- ( |z_\parallel |^2 + x_3^2 )^{1/2} ,z_\parallel + x_\parallel, 0 ,v) \dd v\frac{ \dd z_\parallel}{(|z_\parallel |^2 + x_3^2)^{1/2}}
\\ & + \delta_{k3 }  \int_{ \sqrt{t^2 - |z_\parallel |^2 } > x_3}  \int_{\R^3}   \left( \delta_{i3 } -\frac{(\o_i + \hat{v}_i)\hat{v}_3}{1+ \hat v \cdot \o } \right) f (t- ( |z_\parallel |^2 + x_3^2 )^{1/2} ,z_\parallel + x_\parallel, 0 ,v) \dd v \   \left( \frac{ -x_3}{(|z_\parallel |^2 + x_3^2)^{3/2}} \right)  \dd z_\parallel
\\ &      - \delta_{k3}  \int_{ \sqrt{t^2 - |z_\parallel |^2 } > x_3}  \int_{\R^3} \left( \delta_{i3} - \frac{(\o_i + \hat{v}_i)\hat{v}_3}{1+ \hat v \cdot \o } \right)  \left(  \hat v \cdot \nabla_x f (t- ( |z_\parallel | + x_3^2 )^{1/2} ,z_\parallel + x_\parallel, 0 ,v) \right) \dd v \frac{ x_3 }{(|z_\parallel |^2 + x_3^2)} \dd z_\parallel 
\\ & +  \delta_{k3}  \int_{ \sqrt{t^2 - |z_\parallel |^2 } > x_3}  \int_{\R^3}  \nabla_v  \left( \delta_{i3} - \frac{(\o_i + \hat{v}_i)\hat{v}_3}{1+ \hat v \cdot \o } \right) (E + \hat v \times B - g \mathbf{e}_3 )  \frac{ x_3 f (t- ( |z_\parallel | + x_3^2 )^{1/2} ,z_\parallel + x_\parallel, 0 ,v) }{(|z_\parallel |^2 + x_3^2)} \dd v \dd z_\parallel
\\ & -   \delta_{k3}   \int_{ \sqrt{t^2 - |z_\parallel |^2 } = x_3}   \int_{\R^3} \left( \delta_{i3} - \frac{(\o_i + \hat{v}_i)\hat{v}_3}{1+ \hat v \cdot \o }  \right) f (t- ( |z_\parallel |^2 + x_3^2 )^{1/2} ,z_\parallel + x_\parallel, 0 ,v) \dd v  \frac{x_3 }{\sqrt{t^2 - x_3^2} }   \frac{ \dd S_{ z_\parallel } }{t}.
\end{split} 
\Ee 
Similarly,
\Be \label{Eestbdryneg21}
\begin{split}
& \frac{ \p }{\p_{x_k } } \ref{Eestbdryneg}
\\ = & -\iota_i ( 1 - \delta_{k3} )   \int_{ \sqrt{t^2 - |z_\parallel |^2 } < x_3}  \int_{\R^3}   \left( \delta_{i3 } -\frac{(\bar \o_i + \hat{v}_i)\hat{v}_3}{1+ \hat v \cdot \bar \o } \right)  \p_{x_k } f (t- ( |z_\parallel |^2 + x_3^2 )^{1/2} ,z_\parallel + x_\parallel, 0 ,v) \dd v\frac{ \dd z_\parallel}{(|z_\parallel |^2 + x_3^2)^{1/2}}
\\ &  -\iota_i  \delta_{k3 }  \int_{ \sqrt{t^2 - |z_\parallel |^2 } < x_3}  \int_{\R^3}   \left( \delta_{i3 } -\frac{(\bar \o_i + \hat{v}_i)\hat{v}_3}{1+ \hat v \cdot \bar \o } \right) f (t- ( |z_\parallel |^2 + x_3^2 )^{1/2} ,z_\parallel + x_\parallel, 0 ,v) \dd v \   \left( \frac{-x_3}{(|z_\parallel |^2 + x_3^2)^{3/2}} \right)  \dd z_\parallel
\\ &   + \iota_i  \delta_{k3}  \int_{ \sqrt{t^2 - |z_\parallel |^2 } < x_3}  \int_{\R^3} \left( \delta_{i3 } -\frac{(\bar \o_i + \hat{v}_i)\hat{v}_3}{1+ \hat v \cdot \bar \o } \right) \left(  \hat v \cdot \nabla_x f (t- ( |z_\parallel | + x_3^2 )^{1/2} ,z_\parallel + x_\parallel, 0 ,v) \right) \dd v \frac{ x_3 }{(|z_\parallel |^2 + x_3^2)} \dd z_\parallel 
\\ &  -\iota_i  \delta_{k3}  \int_{ \sqrt{t^2 - |z_\parallel |^2 } > x_3}  \int_{\R^3}  \nabla_v   \left( \delta_{i3 } -\frac{(\bar \o_i + \hat{v}_i)\hat{v}_3}{1+ \hat v \cdot \bar \o } \right) (E + \hat v \times B - g \mathbf{e}_3 ) \frac{x_3  f (t- ( |z_\parallel | + x_3^2 )^{1/2} ,z_\parallel + x_\parallel, 0 ,v) }{(|z_\parallel |^2 + x_3^2)} \dd v  \dd z_\parallel
\\ &  - \iota_i   \delta_{k3}   \int_{ \sqrt{t^2 - |z_\parallel |^2 } = x_3}   \int_{\R^3} \left( \delta_{i3} -  \frac{( \bar \o_i + \hat{v}_i)\hat{v}_3}{1+ \hat v \cdot \bar \o }  \right) f (t- ( |z_\parallel |^2 + x_3^2 )^{1/2} ,z_\parallel + x_\parallel, 0 ,v) \dd v  \frac{x_3 }{\sqrt{t^2 - x_3^2} }   \frac{ \dd S_{ z_\parallel } }{t}.
\end{split} 
\Ee 

Next, by using the change of variables $z = y - x$ and spherical coordinate for $z$ for \ref{Eestinitialpos} and \ref{Eestinitialneg}, we have
\Be \label{pxEestinitialpos}
\begin{split}
 \frac{ \p}{\p_{x_k } }   \ref{Eestinitialpos}  =  & -   \int_{ 0}^{\arccos(-x_3/t ) }   \int_0^{2\pi}  \sum_j \o_j  \left(\delta_{ij} - \frac{(\o_i + \hat{v}_i)\hat{v}_j}{1+ \hat v \cdot \o }\right) \p_{x_k}  f(0, x+z ,v)   ( t  \sin \phi )  \,  \dd v \dd \theta \dd \phi 
\\ &  - \delta_{k3}   \int_0^{2\pi}  \sum_j \o_j  \left(\delta_{ij} - \frac{(\o_i + \hat{v}_i)\hat{v}_j}{1+ \hat v \cdot \o }\right)   f(0, z_\parallel +x_\parallel, 0 ,v)  \,  \dd v \dd \theta,
\end{split}
\Ee
\Be \label{pxEestinitialneg}
\begin{split}
\frac{ \p}{\p_{x_k } }   \ref{Eestinitialneg}    =  &  \iota_i \iota_k \int_{\arccos(-x_3/t ) }^0   \int_0^{2\pi}  \bar \o_j \left(\delta_{ij} - \frac{(\bar \o_i + \hat{v}_i)\hat{v}_j}{1+ \hat v \cdot \bar \o }\right) \p_{x_k}  f(0, \bar x + \bar z ,v)   ( t  \sin \phi )  \,  \dd v \dd \theta \dd \phi 
\\ &  - \iota_i \delta_{k3}   \int_0^{2\pi}  \bar \o_j  \left(\delta_{ij} - \frac{(\bar \o_i + \hat{v}_i)\hat{v}_j}{1+ \hat v \cdot \bar \o }\right)   f(0, z_\parallel +x_\parallel, 0 ,v)  \,  \dd v \dd \theta.
\end{split}
\Ee

Finally, taking $ \frac{\p}{\p x_k }  \ref{Eest3bdrycontri}$, using the equation \ref{VMfrakF}, and then integration by parts in $v$ and we have
\Be
\begin{split}
&  \frac{\p}{\p x_k }  \ref{Eest3bdrycontri} 
\\ & =  - 2( 1 - \delta_{k3} )  \delta_{i3}  \int_{ \sqrt{t^2 - |z_\parallel |^2 } > x_3}  \int_{\R^3}      \p_{x_k } f (t- ( |z_\parallel |^2 + x_3^2 )^{1/2} ,z_\parallel + x_\parallel, 0 ,v)   \dd v\frac{ \dd z_\parallel}{(|z_\parallel |^2 + x_3^2)^{1/2}}
\\ & - 2\delta_{k3 }  \delta_{i3}    \int_{ \sqrt{t^2 - |z_\parallel |^2 } > x_3}  \int_{\R^3}    f (t- ( |z_\parallel |^2 + x_3^2 )^{1/2} ,z_\parallel + x_\parallel, 0 ,v) \dd v  \left( \frac{ -x_3}{(|z_\parallel |^2 + x_3^2)^{3/2}} \right)    \dd z_\parallel
\\ & - 2 \delta_{k3}  \delta_{i3}    \int_{ \sqrt{t^2 - |z_\parallel |^2 } > x_3}  \int_{\R^3}     \hat v \cdot \nabla_x f (t- ( |z_\parallel |^2 + x_3^2 )^{1/2} ,z_\parallel + x_\parallel, 0 ,v) \dd v\frac{  x_3 }{(|z_\parallel |^2 + x_3^2)}  \dd z_\parallel
\\ & +  2  \delta_{k3}   \delta_{i3}   \int_{ 0}^{2 \pi }    \int_{\R^3} f (0 ,z_\parallel + x_\parallel, 0 ,v)  \frac{  x_3 }{t} \dd v    d\theta.
\end{split}
\Ee

\subsection{Derivatives of the Magnetic field in a half space}
We calculate the derivatives of our representation of $B$. We first compute $\p_t B$ by taking $\frac{\p}{\p_t }$ to each term from \ref{Besttat0pos}--\ref{Bestbdrycontri}. Similar to \ref{Eesttat0posest1}, we have
\Be \label{Besttat0posest1}
\begin{split}
& \frac{ \p }{\p_t} \ref{Besttat0pos}  + \frac{ \p }{\p_t} \ref{Besttat0neg} 
\\ = &  \frac{1}{4\pi}  \int_0^{\arccos( -x_3 /t ) } \int_0^{2 \pi } \big(  \p_t  B_{0,i}( x+ z  )+ t  \nabla \p_t B_{0,i}( x+ z  ) \cdot \frac{z}{t}  + \nabla B_{0,i} (x+z) \cdot \frac{z}{t}
\\ & \quad \quad \quad \quad \quad  \quad \quad \quad \quad \quad +\frac{z}{t}  \cdot \nabla^2  B_{0,i} (x+z) \cdot z + \nabla B_0(x+z) \cdot \frac{z}{t}  \big) \sin \phi d\theta d\phi
\\ & + \frac{1}{4\pi } \int_0^{2 \pi }  \left( t \p_t  B_{0,i}( x_\parallel+ z_\parallel,0 ) +  B_{0,i} (x_\parallel+ z_\parallel,0) + \nabla  B_{0,i} (x_\parallel+ z_\parallel,0) \cdot z \right)  \frac{x_3}{t^2 }  d\theta
\\ & +  \frac{\iota_i}{4\pi}  \int_{\arccos( -x_3 /t ) }^{\pi } \int_0^{2 \pi } \big(  \p_t  B_{0,i}( \bar x + \bar z  )+ t  \nabla \p_t B_{0,i}( \bar x + \bar z  ) \cdot \frac{\bar z}{t}  + \nabla B_{0,i} (\bar x + \bar z) \cdot \frac{\bar z}{t}
\\ & \quad \quad \quad \quad \quad  \quad \quad \quad \quad \quad +\frac{\bar z}{t}  \cdot \nabla^2  B_{0,i} (\bar x + \bar z) \cdot \bar z + \nabla B_0(\bar x + \bar z) \cdot \frac{\bar z}{t}  \big) \sin \phi d\theta d\phi
\\ & - \frac{\iota_i }{4\pi } \int_0^{2 \pi }  \left( t \p_t  B_{0,i}( x_\parallel+  z_\parallel,0 ) +  B_{0,i} (x_\parallel+  z_\parallel,0) + \nabla  B_{0,i} (x_\parallel+  z_\parallel,0) \cdot \bar z \right)  \frac{x_3}{t^2 }  d\theta,
\end{split}
\Ee
where $z = \begin{bmatrix} t sin \phi \cos \theta \\ t \sin \phi \sin \theta \\ t \cos \phi \end{bmatrix} $. Similar to \ref{ptEestbulkpos11}, and \ref{ptEestbulkneg11}, we get
\Be \label{ptBestbulkpos11}
\begin{split}
\frac{\p}{\p t }  \ref{Bestbulkpos}  = &  \int_{ \{ |z| < t \}  \cap \{x_3 + z_3 > 0\}}   \int_{\R^3}  \frac{ (|\hat{v}|^2-1  )( \o \times \hat v )_i }{|z|^2 (1+  \hat v \cdot \o )^2}    \hat v \cdot \nabla_x f(t-|z|,x+z ,v)\dd v \dd z
 \\ & - \int_{ \{ |z| < t \}  \cap \{x_3 + z_3 > 0\}}  \int_{\R^3} \nabla_v  \left( \frac{ (|\hat{v}|^2-1  )( \o \times \hat v )_i }{ (1+  \hat v \cdot \o )^2} \right)  \cdot (E + \hat v \times B - g \mathbf{e}_3 ) \frac{ f (t-|z|,x+z ,v)}{|z|^2 } \dd v \dd z
\\ & - \int_{ \{ |z| = t \}  \cap \{x_3 + z_3 > 0\}}  \int_{\R^3}\frac{ (|\hat{v}|^2-1  ) ( \o \times \hat v )_i }{t^2 (1+  \hat v \cdot \o )^2} f(0,x+z ,v)\dd v \dd S_z.
\end{split}
\Ee
And by direct calculation,
\Be
\begin{split}
& \nabla_v \left( \frac{ (\o \times \hat v)_i ( 1- |\hat v |^2 ) }{(1 + \hat v \cdot \o )^2 } \right) 
\\ & = \frac{ \nabla_v [ (\o \times \hat v )_i ] }{ (1 +|v|^2 )(1+ \hat v \cdot \o )^2 } +  \frac{ - 2 v (\o \times \hat v )_i  }{ ( 1 +|v|^2 )^2  ( 1 +\hat v \cdot \o )^2 }   +  \frac{2}{ (1 + \hat v \cdot \o )^3 } \left(  \frac{ (\o - (\o \cdot \hat v ) \hat v ) ( \o \times \hat v  )_i }{ \langle v \rangle^3 } \right).
\end{split}
\Ee
And
\Be \label{ptBestbulkneg11}
\begin{split}
\frac{\p}{\p t }  \ref{Bestbulkneg}  = & \iota_i \int_{ \{ |z| < t \}  \cap \{x_3 + z_3 < 0\}}   \int_{\R^3}  \frac{ (|\hat{v}|^2-1  )( \bar \o \times \hat v )_i }{|z|^2 (1+  \hat v \cdot \bar \o )^2}    \hat v \cdot \nabla_x f(t-|z|,\bar x+ \bar z ,v)\dd v \dd z
 \\ & - \iota_i \int_{ \{ |z| < t \}  \cap \{x_3 + z_3 < 0\}}  \int_{\R^3} \nabla_v  \left( \frac{ (|\hat{v}|^2-1  )( \bar \o \times \hat v )_i }{ (1+  \hat v \cdot \bar \o )^2} \right)  \cdot (E + \hat v \times B - g \mathbf{e}_3 ) \frac{ f (t-|z|,\bar x+ \bar z ,v)}{|z|^2 } \dd v \dd z
\\ & - \iota_i \int_{ \{ |z| = t \}  \cap \{x_3 + z_3 < 0\}}  \int_{\R^3}\frac{ (|\hat{v}|^2-1  )( \bar \o \times \hat v )_i }{t^2 (1+  \hat v \cdot \bar \o )^2} f(0,\bar x+ \bar z ,v)\dd v \dd S_z.
\end{split}
\Ee

Next, similar to \ref{ptEestSpos11} and \ref{ptEestSneg11},
\Be \label{ptBestSpos11}
\begin{split}
& \frac{ \p}{\p t } \ref{BestSpos}
\\  = &  \int_{ \{ |z| < t \} \cap \{ z_3+x_3 > 0 \} } \int_{\mathbb R^3 }   a^B_i (v ,\o ) \cdot (E + \hat v \times B - g\mathbf e_3 ) \hat v \cdot \nabla_x f  (t-|z|,x+z ,v) \dd v \frac{ \dd z}{|z|}
\\ & - \int_{ \{ |z| < t \} \cap \{ z_3+x_3 > 0 \} } \int_{\mathbb R^3 } \nabla_v \left(  a^B_i (v ,\o ) \cdot (E + \hat v \times B - g\mathbf e_3 ) \right) \cdot (E + \hat v \times B - g \mathbf{e}_3 )  f  (t-|z|,x+z ,v) \dd v \frac{ \dd z}{|z|}  
\\ & -  \int_{ \{ |z| < t \} \cap \{ z_3+x_3 > 0 \} } \int_{\mathbb R^3 } a^B_i (v ,\o ) \cdot (\p_t E + \hat v \times \p_t B - g\mathbf e_3 )  f (t-|z|,x+z ,v) \dd v \frac{ \dd z}{|z|} 
\\ & - \int_{ \{ |z| = t \} \cap \{ z_3+x_3 > 0 \} } \int_{\mathbb R^3 } a^B_i (v ,\o ) \cdot (E + \hat v \times B - g\mathbf e_3 )   f (0,x+z ,v) \dd v \frac{ \dd S_z}{t}.
\end{split}
\Ee
\Be \label{ptBestSneg11}
\begin{split}
& \frac{ \p}{\p t } \ref{BestSneg}
\\   = & \iota_i \int_{ \{ |z| < t \} \cap \{ z_3+x_3 < 0 \} } \int_{\mathbb R^3 }   a^B_i (v ,\bar \o ) \cdot (E + \hat v \times B - g\mathbf e_3 ) \hat v \cdot \nabla_x f  (t-|z|,\bar x + \bar z ,v) \dd v \frac{ \dd z}{|z|}
\\ & - \iota_i \int_{ \{ |z| < t \} \cap \{ z_3+x_3 < 0 \} } \int_{\mathbb R^3 } \nabla_v \left(  a^B_i (v ,\bar \o ) \cdot (E + \hat v \times B - g\mathbf e_3 ) \right) \cdot (E + \hat v \times B - g \mathbf{e}_3 )  f  (t-|z|,\bar x + \bar z ,v) \dd v \frac{ \dd z}{|z|}  
\\ & - \iota_i  \int_{ \{ |z| < t \} \cap \{ z_3+x_3 < 0 \} } \int_{\mathbb R^3 } a^B_i (v ,\bar \o ) \cdot (\p_t E + \hat v \times \p_t B - g\mathbf e_3 )  f (t-|z|,\bar x + \bar z ,v) \dd v \frac{ \dd z}{|z|} 
\\ & - \iota_i \int_{ \{ |z| = t \} \cap \{ z_3+x_3 < 0 \} } \int_{\mathbb R^3 } a^B_i (v ,\bar \o ) \cdot (E + \hat v \times B - g\mathbf e_3 )   f (0,\bar x + \bar z ,v) \dd v \frac{ \dd S_z}{t}.
\end{split}
\Ee

Next, similar to \ref{ptEestbdrypos11} and \ref{ptEestbdryneg11}, we have
\Be \label{ptBestbdrypos11}
\begin{split}
& \frac{ \p }{\p t } \ref{Bestbdrypos}
\\ = &  \int_{ |z_\parallel | < \sqrt{ t^2 - x_3^2 } }  \int_{\R^3} \left(  -(e_3 \times \hat v)_i   + \frac{ ( \o \times \hat v)_i   \hat v_3 }{1+ \hat{v} \cdot  \o} \right) \frac{\hat v \cdot \nabla_x f (t- ( |z_\parallel |^2 + x_3^2 )^{1/2} ,z_\parallel + x_\parallel, 0 ,v) }{ (|z_\parallel |^2 + x_3^2)^{1/2}} \dd v \dd z_\parallel
 \\ & + \int_{ |z_\parallel | < \sqrt{ t^2 - x_3^2 } }  \int_{\R^3} \nabla_v \left(  -(e_3 \times \hat v)_i   + \frac{ ( \o \times \hat v)_i   \hat v_3 }{1+ \hat{v} \cdot  \o} \right) \cdot  ( E + \hat v \times B -  g \mathbf e_3 )
 \\ & \quad \quad \quad \quad \quad \quad \quad \quad \quad \times \frac{ f (t- ( |z_\parallel |^2 + x_3^2 )^{1/2} ,z_\parallel + x_\parallel, 0 ,v) }{ (|z_\parallel |^2 + x_3^2)^{1/2}} \dd v \dd z_\parallel
\\ & + \int_0^{2\pi }   \int_{\R^3} \left(  -(e_3 \times \hat v)_i   + \frac{ ( \o \times \hat v)_i   \hat v_3 }{1+ \hat{v} \cdot  \o} \right) f ( 0  ,z_\parallel + x_\parallel, 0 ,v)  \dd v \dd \theta.
\end{split}
\Ee
\Be \label{ptBestbdryneg11}
\begin{split}
& \frac{ \p }{\p t } \ref{Bestbdryneg}
\\ = &  \iota_i \int_{ |z_\parallel | < \sqrt{ t^2 - x_3^2 } }  \int_{\R^3} \left(  -(e_3 \times \hat v)_i   + \frac{ ( \bar \o \times \hat v)_i   \hat v_3 }{1+ \hat{v} \cdot  \bar \o} \right) \frac{\hat v \cdot \nabla_x f (t- ( |z_\parallel |^2 + x_3^2 )^{1/2} ,z_\parallel + x_\parallel, 0 ,v) }{ (|z_\parallel |^2 + x_3^2)^{1/2}} \dd v \dd z_\parallel
 \\ &  + \iota_i  \int_{ |z_\parallel | < \sqrt{ t^2 - x_3^2 } }  \int_{\R^3} \nabla_v \left(   -(e_3 \times \hat v)_i   + \frac{ ( \bar \o \times \hat v)_i   \hat v_3 }{1+ \hat{v} \cdot  \bar \o}   \right) \cdot  ( E + \hat v \times B -  g \mathbf e_3 )
  \\ & \quad \quad \quad \quad \quad \quad \quad \quad \quad \times  \frac{ f (t- ( |z_\parallel |^2 + x_3^2 )^{1/2} ,z_\parallel + x_\parallel, 0 ,v) }{ (|z_\parallel |^2 + x_3^2)^{1/2}} \dd v \dd z_\parallel
\\ &  + \iota_i  \int_0^{2\pi }   \int_{\R^3} \left(  -(e_3 \times \hat v)_i   + \frac{ ( \bar \o \times \hat v)_i   \hat v_3 }{1+ \hat{v} \cdot  \bar \o}   \right) f ( 0  ,z_\parallel + x_\parallel, 0 ,v)  \dd v \dd \theta.
\end{split}
\Ee

Next, similar to \ref{ptEestinitialpos11} and \ref{ptEestinitialneg11}, we have
Thus
\Be \label{ptBestinitialpos11}
\begin{split}
\frac{\p}{\p t }  \ref{Bestinitialpos} = &   \int_0^{\arccos(-x_3/t )} \int_0^{2\pi}  \int_{\R^3}  \left(  \frac{  (  \o \times \hat v)_i   }{1+ \hat{v} \cdot  \o} \right)  f(0,x+z ,v) \sin \phi  \dd v  d \theta d \phi
\\ & +   \int_0^{\arccos(-x_3/t )} \int_0^{2\pi}  \int_{\R^3}   \left(  \frac{  (  \o \times \hat v)_i   }{1+ \hat{v} \cdot  \o} \right)  ( \nabla_x f(0,x+z ,v) \cdot z ) \sin \phi  \dd v  d \theta d \phi
\\ & +  \int_0^{2\pi}  \int_{\R^3}   \sum_j \o_j  \left(  \frac{  (  \o \times \hat v)_i   }{1+ \hat{v} \cdot  \o} \right)    f(0,x_\parallel +z_\parallel, 0 ,v) \frac{x_3}{ t}  \  \dd v  d \theta.
\end{split}
\Ee
\Be \label{ptBestinitialneg11}
\begin{split}
\frac{\p}{\p t }  \ref{Bestinitialneg} = & \iota_i  \int_{\arccos(-x_3/t )}^0 \int_0^{2\pi}  \int_{\R^3}   \left(  \frac{  (  \bar \o \times \hat v)_i   }{1+ \hat{v} \cdot  \bar \o} \right)   f(0,\bar x+ \bar z ,v) \sin \phi  \dd v  d \theta d \phi
\\ & +  \iota_i  \int_{\arccos(-x_3/t )}^0 \int_0^{2\pi}  \int_{\R^3}    \left(  \frac{  (  \bar \o \times \hat v)_i   }{1+ \hat{v} \cdot  \bar \o} \right) ( \nabla_x f(0,\bar x+ \bar z ,v) \cdot \bar z ) \sin \phi  \dd v  d \theta d \phi
\\ & -  \iota_i  \int_0^{2\pi}  \int_{\R^3}    \left(  \frac{  (  \bar \o \times \hat v)_i   }{1+ \hat{v} \cdot  \bar \o} \right)   f(0,x_\parallel +z_\parallel, 0 ,v)  \frac{x_3}{t }  \  \dd v  d \theta.
\end{split}
\Ee

Finally, similar to \ref{ptEest6}, we have
\Be \label{ptBest6}
\begin{split}
& \frac{ \p}{\p t } \ref{Bestbdrycontri}  
 \\  = & - 2 (1- \delta_{i3}  )  \int_{ |z_\parallel | < \sqrt{ t^2 - x_3^2 } }  \int_{\R^3}  \frac{ \hat v_{\overline{i +1}} \hat v \cdot \nabla_x f (t- ( |z_\parallel |^2 + x_3^2 )^{1/2} ,z_\parallel + x_\parallel, 0 ,v) }{ (|z_\parallel |^2 + x_3^2)^{1/2}} \dd v \dd z_\parallel
 \\ & - 2 (1- \delta_{i3}  )  \int_{ |z_\parallel | < \sqrt{ t^2 - x_3^2 } }  \int_{\R^3}  \frac{  \nabla_v(\hat v_{\overline{i +1}} ) \cdot (E + \hat v \times B - g \mathbf e_3)  f (t- ( |z_\parallel |^2 + x_3^2 )^{1/2} ,z_\parallel + x_\parallel, 0 ,v) }{ (|z_\parallel |^2 + x_3^2)^{1/2}} \dd v \dd z_\parallel
\\ & - 2(1- \delta_{i3} ) \int_0^{2\pi }   \int_{\R^3}   \hat v_{\overline{i +1}} f ( 0  ,z_\parallel + x_\parallel, 0 ,v)  \dd v \dd \theta.
\end{split}
\Ee

Next, we calculate the spatial derivatives to $B(t,x)$ by taking $\frac{\p}{\p_{x_k}} $ to \ref{Besttat0pos}-\ref{Bestbdrycontri}. Similar to \ref{pxEesttat0pos}, we have
\Be \label{pxBesttat0pos}
\begin{split}
& \frac{ \p}{\p_{x_k } }  \ref{Besttat0pos} + \frac{ \p}{\p_{x_k } }  \ref{Besttat0neg}  
\\ = &   \frac{1}{4\pi } \int_0^{\arccos(-x_3/t ) }  \int_{0 }^{2\pi}    \left( t \p_{x_k } \p_t  B_{0,i}( x + z  ) +  \p_{x_k }  B_{0,i} (x+z) +  \nabla  \p_{x_k }  B_{0,i} ( x+ z) \cdot z \right)  \sin \phi d\theta d \phi
\\ & + \frac{1}{4 \pi t}  \delta_{k3 } \int_0^{2\pi } (  t \p_t  B_{0,i}( x_\parallel + z_\parallel , 0  ) +  B_{0,i} (x_\parallel + z_\parallel , 0) + \nabla  B_{0,i} ( x_\parallel + z_\parallel , 0) \cdot z )  d\theta
\\ & -  \frac{\iota_i \iota_k}{4\pi}  \int_{\arccos( -x_3 /t ) }^{\pi } \int_0^{2 \pi } \big(   t \p_{x_k } \p_t  B_{0,i}(\bar x + \bar z  ) +  \p_{x_k }  B_{0,i} (\bar x + \bar z) +  \nabla  \p_{x_k }  B_{0,i} ( \bar x + \bar z) \cdot \bar z  \big) \sin \phi d\theta d\phi
\\ & + \frac{\iota_i }{4\pi t  } \delta_{k3 }  \int_0^{2 \pi }   \left( t \p_t  B_{0,i}( x_\parallel+  z_\parallel,0 ) +  B_{0,i} (x_\parallel+  z_\parallel,0) + \nabla  B_{0,i} (x_\parallel+  z_\parallel,0) \cdot \bar z \right)   d\theta.
\end{split}
\Ee
Then, similar to \ref{Eestbulkpos12} and \ref{Eestbulkneg12}, we have
\Be \label{Bestbulkpos12}
\begin{split}
\frac{ \p}{\p_{x_k } }  \ref{Bestbulkpos} = & \int_{ \{ |z| < t  \} \cap \{z_3 +x_3 > 0 \} } \int_{\R^3}  \frac{   ( \o \times \hat v  )_i \left( 1   -  |\hat v |^2  \right) }{( 1 + \hat v \cdot \o )^2|z |^2}  \p_{x_k}f (t-|z| ,x+z ,v) \dd v \dd z  
 \\& + \delta_{k3} \int_{ \{  (| z_\parallel |^2 + |x_3|^2 )^{1/2} < t  \}  }   \int_{\R^3}  \frac{   ( \o \times \hat v  )_i \left( 1   -  |\hat v |^2  \right) }{( 1 + \hat v \cdot \o )^2 (| z_\parallel |^2 + |x_3|^2 )}  f(t-( |z_\parallel  |^2 + x_3^2 )^{1/2}, x_\parallel + z_\parallel, 0  ,v) \dd v \dd z_\parallel,
\end{split}
\Ee
\Be \label{Bestbulkneg12}
\begin{split}
\frac{ \p}{\p_{x_k } }  \ref{Bestbulkneg} = & +  \iota_i \iota_k \int_{ \{ |z| < t  \} \cap \{x_3+z_3 < 0 \} } \int_{\R^3}   \frac{   ( \bar \o \times \hat v  )_i \left( 1   -  |\hat v |^2  \right) }{( 1 + \hat v \cdot \bar \o )^2|z |^2} \p_{x_k}f (t-|z| , \bar x+\bar z ,v) \dd v \dd z  
 \\& - \iota_i \delta_{k3} \int_{ \{  (| z_\parallel |^2 + |x_3|^2 )^{1/2} < t  \}  }   \int_{\R^3}    \frac{   ( \bar \o \times \hat v  )_i \left( 1   -  |\hat v |^2  \right) }{( 1 + \hat v \cdot \bar \o )^2(|z_\parallel |^2 + x_3^2 )}  f(t-( |z_\parallel  |^2 + x_3^2 )^{1/2}, x_\parallel + z_\parallel , 0  ,v) \dd v \dd z_\parallel.
\end{split}
\Ee

Next, similar to \ref{EestSpos12} and \ref{EestSneg2}, we have
\Be \label{BestSpos12}
\begin{split}
\frac{ \p }{\p_{x_k } } \ref{BestSpos} = &  \int_{ \{ |z| < t  \} \cap \{z_3 +x_3 > 0 \} }  \int_{\R^3}   a^B_i (v ,\o ) \cdot ( \p_{x_k } E +   \hat v  \times  \p_{x_k } B)   f (t-|z|,x+z ,v) \dd v \frac{ \dd z}{|z|}
\\  & +  \int_{ \{ |z| < t  \} \cap \{z_3 +x_3 > 0 \} } \int_{\R^3}   a^B_i (v ,\o ) \cdot (  E +   \hat v  \times B) \p_{x_k}  f (t-|z|,x+z ,v) \dd v \frac{ \dd z}{|z|}
\\ & + \delta_{k3} \int_{ \{  (| z_\parallel |^2 + |x_3|^2 )^{1/2} < t  \}  } \int_{\R^3}   a^B_i (v ,\o ) \cdot (  E +   \hat v  \times B) \frac{f (t-(| z_\parallel |^2 + |x_3|^2 )^{1/2},x_\parallel + z_\parallel, 0 ,v)  }{(| z_\parallel |^2 + |x_3|^2 )^{1/2}}    \dd v \dd z_\parallel,
\end{split}
\Ee
and
\Be \label{BestSneg2}
\begin{split}
\frac{ \p }{\p_{x_k } } \ref{BestSneg} = &  \iota_i \int_{ \{ |z| < t  \} \cap \{z_3 +x_3 < 0 \} }  \int_{\R^3}   a^B_i (v ,\bar \o ) \cdot ( \p_{x_k } E +   \hat v  \times  \p_{x_k } B)   f (t-|z|,\bar x + \bar z ,v) \dd v \frac{ \dd z}{|z|}
\\  & + \iota_i \iota_k  \int_{ \{ |z| < t  \} \cap \{z_3 +x_3 < 0 \} }  \int_{\R^3}   a^B_i (v ,\bar \o ) \cdot (  E +   \hat v  \times B) \p_{x_k}  f (t-|z|,\bar x + \bar z ,v) \dd v \frac{ \dd z}{|z|}
\\ & - \iota_i \delta_{k3} \int_{ \{  (| z_\parallel |^2 + |x_3|^2 )^{1/2} < t  \}  } \int_{\R^3}   a^B_i (v ,\bar \o ) \cdot (  E +   \hat v  \times B) \frac{  f (t-(| z_\parallel |^2 + |x_3|^2 )^{1/2},x_\parallel+z_\parallel , 0 ,v)   }{(| z_\parallel |^2 + |x_3|^2 )^{1/2}}  \dd v  \dd z_\parallel.
\end{split}
\Ee

Next, similar to \ref{Eestbdrypos21} and \ref{Eestbdryneg21}, we have
\Be \label{Bestbdrypos21}
\begin{split}
& \frac{ \p }{\p_{x_k } } \ref{Bestbdrypos}
\\ = &  ( 1 - \delta_{k3} )   \int_{ \sqrt{t^2 - |z_\parallel |^2 } > x_3}  \int_{\R^3}   \left( -(e_3 \times \hat v)_i   + \frac{ ( \o \times \hat v)_i   \hat v_3 }{1+ \hat{v} \cdot  \o}   \right)  \p_{x_k } \frac{ f (t- ( |z_\parallel |^2 + x_3^2 )^{1/2} ,z_\parallel + x_\parallel, 0 ,v)}{(|z_\parallel |^2 + x_3^2)^{1/2}} \dd v \dd z
\\ & + \delta_{k3 }  \int_{ \sqrt{t^2 - |z_\parallel |^2 } > x_3}  \int_{\R^3}    \left( -(e_3 \times \hat v)_i   + \frac{ ( \o \times \hat v)_i   \hat v_3 }{1+ \hat{v} \cdot  \o}   \right) \frac{-x_3 f (t- ( |z_\parallel |^2 + x_3^2 )^{1/2} ,z_\parallel + x_\parallel, 0 ,v) }{(|z_\parallel |^2 + x_3^2)^{3/2}} \dd v   \dd z_\parallel
\\ &      - \delta_{k3}  \int_{ \sqrt{t^2 - |z_\parallel |^2 } > x_3}  \int_{\R^3}   \left( -(e_3 \times \hat v)_i   + \frac{ ( \o \times \hat v)_i   \hat v_3 }{1+ \hat{v} \cdot  \o}   \right) \left(  \frac{x_3 \hat v \cdot \nabla_x f (t- ( |z_\parallel | + x_3^2 )^{1/2} ,z_\parallel + x_\parallel, 0 ,v)}{(|z_\parallel |^2 + x_3^2)} \right) \dd v \dd z_\parallel 
\\ & +  \delta_{k3}  \int_{ \sqrt{t^2 - |z_\parallel |^2 } > x_3}  \int_{\R^3}  \nabla_v   \left( -(e_3 \times \hat v)_i   + \frac{ ( \o \times \hat v)_i   \hat v_3 }{1+ \hat{v} \cdot  \o}   \right) (E + \hat v \times B - g \mathbf{e}_3 )  
\\ & \quad \quad \quad \quad \quad \quad \quad \quad \quad \quad \quad \times  f (t- ( |z_\parallel | + x_3^2 )^{1/2} ,z_\parallel + x_\parallel, 0 ,v)  \dd v \frac{ x_3 }{(|z_\parallel |^2 + x_3^2)} \dd z_\parallel
\\ & -   \delta_{k3}   \int_{ \sqrt{t^2 - |z_\parallel |^2 } = x_3}   \int_{\R^3}   \left( -(e_3 \times \hat v)_i   + \frac{ ( \o \times \hat v)_i   \hat v_3 }{1+ \hat{v} \cdot  \o}   \right)  f (t- ( |z_\parallel |^2 + x_3^2 )^{1/2} ,z_\parallel + x_\parallel, 0 ,v) \dd v  \frac{x_3 }{\sqrt{t^2 - x_3^2} }   \frac{ \dd S_{ z_\parallel } }{t},
\end{split} 
\Ee 
and
\Be \label{Bestbdryneg21}
\begin{split}
& \frac{ \p }{\p_{x_k } } \ref{Bestbdryneg}
\\ = & \iota_i ( 1 - \delta_{k3} )   \int_{ \sqrt{t^2 - |z_\parallel |^2 } < x_3}  \int_{\R^3}   \left( -(e_3 \times \hat v)_i   + \frac{ ( \bar \o \times \hat v)_i   \hat v_3 }{1+ \hat{v} \cdot  \bar \o}   \right)   \p_{x_k } \frac{ f (t- ( |z_\parallel |^2 + x_3^2 )^{1/2} ,z_\parallel + x_\parallel, 0 ,v)}{(|z_\parallel |^2 + x_3^2)^{1/2}} \dd v \dd z
\\ &  + \iota_i  \delta_{k3 }  \int_{ \sqrt{t^2 - |z_\parallel |^2 } < x_3}  \int_{\R^3}   \left( -(e_3 \times \hat v)_i   + \frac{ ( \bar \o \times \hat v)_i   \hat v_3 }{1+ \hat{v} \cdot  \bar \o}   \right) \frac{-x_3  f (t- ( |z_\parallel |^2 + x_3^2 )^{1/2} ,z_\parallel + x_\parallel, 0 ,v)}{(|z_\parallel |^2 + x_3^2)^{3/2}} \dd v   \dd z_\parallel
\\ &   - \iota_i  \delta_{k3}  \int_{ \sqrt{t^2 - |z_\parallel |^2 } < x_3}  \int_{\R^3}  \left( -(e_3 \times \hat v)_i   + \frac{ ( \bar \o \times \hat v)_i   \hat v_3 }{1+ \hat{v} \cdot  \bar \o}   \right)   \left( \frac{ x_3 \hat v \cdot \nabla_x f (t- ( |z_\parallel | + x_3^2 )^{1/2} ,z_\parallel + x_\parallel, 0 ,v) }{(|z_\parallel |^2 + x_3^2)} \right) \dd v \dd z_\parallel 
\\ & + \iota_i  \delta_{k3}  \int_{ \sqrt{t^2 - |z_\parallel |^2 } > x_3}  \int_{\R^3}  \nabla_v    \left( -(e_3 \times \hat v)_i   + \frac{ ( \bar \o \times \hat v)_i   \hat v_3 }{1+ \hat{v} \cdot  \bar \o}   \right)  (E + \hat v \times B - g \mathbf{e}_3 )   \\ & \quad \quad \quad \quad \quad \quad \quad \quad \quad \quad \quad \times f (t- ( |z_\parallel | + x_3^2 )^{1/2} ,z_\parallel + x_\parallel, 0 ,v)  \dd v \frac{ x_3 }{(|z_\parallel |^2 + x_3^2)} \dd z_\parallel
\\ & +  \iota_i   \delta_{k3}   \int_{ \sqrt{t^2 - |z_\parallel |^2 } = x_3}   \int_{\R^3}  \left( -(e_3 \times \hat v)_i   + \frac{ ( \bar \o \times \hat v)_i   \hat v_3 }{1+ \hat{v} \cdot  \bar \o}   \right)  f (t- ( |z_\parallel |^2 + x_3^2 )^{1/2} ,z_\parallel + x_\parallel, 0 ,v) \dd v  \frac{x_3 }{\sqrt{t^2 - x_3^2} }   \frac{ \dd S_{ z_\parallel } }{t}.
\end{split} 
\Ee

Next, similar to \ref{pxEestinitialpos} and \ref{Eestinitialneg}, we have
\Be \label{pxBestinitialpos}
\begin{split}
 \frac{ \p}{\p_{x_k } }   \ref{Bestinitialpos}   =  &    \int_{ 0}^{\arccos(-x_3/t ) }   \int_0^{2\pi}  \left(  \frac{  (  \o \times \hat v)_i   }{1+ \hat{v} \cdot  \o}\right)   \p_{x_k}  f(0, x+z ,v)   ( t  \sin \phi )  \,  \dd v \dd \theta \dd \phi 
\\ &  + \delta_{k3}   \int_0^{2\pi}  \left(  \frac{  (  \o \times \hat v)_i   }{1+ \hat{v} \cdot  \o}\right)     f(0, z_\parallel +x_\parallel, 0 ,v)  \,  \dd v \dd \theta,
\end{split}
\Ee
\Be \label{pxBestinitialneg}
\begin{split}
 \frac{ \p}{\p_{x_k } }   \ref{Bestinitialneg}   =  &  \iota_i \iota_k \int_{\arccos(-x_3/t ) }^0   \int_0^{2\pi}  \left(  \frac{  (  \bar \o \times \hat v)_i   }{1+ \hat{v} \cdot  \bar \o}\right)   \p_{x_k}  f(0, \bar x + \bar z ,v)   ( t  \sin \phi )  \,  \dd v \dd \theta \dd \phi 
\\ &  - \iota_i \delta_{k3}   \int_0^{2\pi}  \left(  \frac{  (  \bar \o \times \hat v)_i   }{1+ \hat{v} \cdot  \bar \o}\right)    f(0, z_\parallel +x_\parallel, 0 ,v)  \,  \dd v \dd \theta.
\end{split}
\Ee

Finally, we have
\Be
\begin{split}
&  \frac{\p}{\p x_k }  \ref{Bestbdrycontri} 
\\ & =  - 2( 1 - \delta_{k3} )(1-\delta_{i3} )   \int_{ \sqrt{t^2 - |z_\parallel |^2 } > x_3}  \int_{\R^3}   \hat v_{\overline{i+1} }    \p_{x_k } f (t- ( |z_\parallel |^2 + x_3^2 )^{1/2} ,z_\parallel + x_\parallel, 0 ,v)   \dd v\frac{ \dd z_\parallel}{(|z_\parallel |^2 + x_3^2)^{1/2}}
\\ & - 2\delta_{k3 } (1-\delta_{i3} )  \int_{ \sqrt{t^2 - |z_\parallel |^2 } > x_3}  \int_{\R^3}   \hat v_{\overline{i+1} }  f (t- ( |z_\parallel |^2 + x_3^2 )^{1/2} ,z_\parallel + x_\parallel, 0 ,v) \dd v  \left( \frac{ -x_3}{(|z_\parallel |^2 + x_3^2)^{3/2}} \right)    \dd z_\parallel
\\ & - 2 \delta_{k3}  (1-\delta_{i3} )  \int_{ \sqrt{t^2 - |z_\parallel |^2 } > x_3}  \int_{\R^3}    \hat v_{\overline{i+1} }  \hat v \cdot \nabla_x f (t- ( |z_\parallel |^2 + x_3^2 )^{1/2} ,z_\parallel + x_\parallel, 0 ,v) \dd v\frac{  x_3 }{(|z_\parallel |^2 + x_3^2)}  \dd z_\parallel
\\ & + 2 \delta_{k3}  (1-\delta_{i3} )  \int_{ \sqrt{t^2 - |z_\parallel |^2 } > x_3}  \int_{\R^3}  \nabla_v(   \hat v_{\overline{i+1} }  ) \cdot (E + \hat v \times B  - g \mathbf e_3 ) \frac{x_3 f (t- ( |z_\parallel |^2 + x_3^2 )^{1/2} ,z_\parallel + x_\parallel, 0 ,v) }{(|z_\parallel |^2 + x_3^2)} \dd v \dd z_\parallel
\\ & +  2  \delta_{k3} (1-\delta_{i3} )  \int_{ 0}^{2 \pi }    \int_{\R^3}   \hat v_{\overline{i+1} }  f (0 ,z_\parallel + x_\parallel, 0 ,v)  \frac{  x_3 }{t} \dd v    d\theta.
\end{split}
\Ee

   \unhide

\end{document}